\documentclass[a4paper,UKenglish,cleveref, autoref, thm-restate]{lipics-v2021}
\makeatletter
\def\input@path{{./}{../}}
\makeatother

\usepackage{tikz}
\usepackage{tikz-3dplot}   % for a simple orthographic 3‑D view
\usetikzlibrary{arrows.meta,calc,positioning,decorations.pathmorphing,decorations.pathreplacing}
\usetikzlibrary{matrix}

\usepackage{xparse} 
\usepackage{mathtools}
\usepackage{amsmath}
\usepackage{amsthm}
\usepackage{amsfonts}

\usepackage{hyperref}
\usepackage{cleveref}

\usepackage{cite} % sorts and compresses numeric citations
\usepackage{subcaption}

\usepackage[
    createShortEnv,
    disablePatchSection,
    commandRef=Cref
    ]{proof-at-the-end}
    
    \newtheorem{puzzle}[theorem]{Puzzle}
\graphicspath{{../}{./}}

\usepackage{chngcntr}
\counterwithin{theorem}{section}

\counterwithin{lemma}{section}

\counterwithin{definition}{section}

\counterwithin{observation}{section}

\counterwithin{remark}{section}

\counterwithin{puzzle}{section}

\counterwithin{claim}{section}

\counterwithin{corollary}{section}

\usepackage{silence}
\title{The Diameter of (Threshold) Geometric Inhomogeneous Random Graphs} %TODO Please add

\author{Zylan Benjert}{TU Delft, Netherlands}{z.benjert@student.tudelft.nl}{}{}
\author{Kostas Lakis}{ETH Zurich, Switzerland}{konstantinos.lakis@inf.ethz.ch}{}{}
\author{Johannes Lengler}{ETH Zurich, Switzerland}{johannes.lengler@inf.ethz.ch}{}{}
\author{Raghu Raman Ravi}{ETH Zurich, Switzerland}{raghu.ravi@inf.ethz.ch}{}{}

\authorrunning{Z. Benjert, K. Lakis, J. Lengler and R. R. Ravi} %TODO mandatory. First: Use abbreviated first/middle names. Second (only in severe cases): Use first author plus 'et al.'

\Copyright{Zylan Benjert, Kostas Lakis, Johannes Lengler and Raghu Raman Ravi} %TODO mandatory, please use full first names. LIPIcs license is "CC-BY";  http://creativecommons.org/licenses/by/3.0/


 %TODO mandatory; please add comma-separated list of keywords

\relatedversion{} %optional, e.g. full version hosted on arXiv, HAL, or other respository/website
\EventEditors{John Q. Open and Joan R. Access}
\EventNoEds{2}
\EventLongTitle{42nd Conference on Very Important Topics (CVIT 2016)}
\EventShortTitle{CVIT 2016}
\EventAcronym{CVIT}
\EventYear{2016}
\EventDate{December 24--27, 2016}
\EventLocation{Little Whinging, United Kingdom}
\EventLogo{}
\SeriesVolume{42}
\ArticleNo{23}
\newcommand{\density}{\ensuremath{\lambda}}
\newcommand{\groundspace}{\ensuremath{\mathcal{V}}}
\newcommand{\weightgeometry}{\ensuremath{\mathcal{V}_w}}
\newcommand{\Dzero}{\ensuremath{D_{0}}}
\newcommand{\Di}[1]{\ensuremath{D_{#1}}}
\newcommand{\tessLevel}[1]{\ensuremath{\mathcal{T}_{#1}}}
\newcommand{\tessPower}[1]{\ensuremath{e_{#1}}}
\newcommand{\weightof}[1]{\ensuremath{w_{#1}}}
\newcommand{\geomdist}[2]{\ensuremath{|#1 - #2|}}
\newcommand{\volbetween}[2]{\geomdist{#1}{#2}^d}
\newcommand{\parentbox}[1]{\pi\!\left(#1\right)}

\newcommand{\boxof}[1]{\ensuremath{B_{#1}}}
\newcommand{\Lof}[2]{\ensuremath{L(#1, #2)}}
\newcommand{\Lprimeof}[2]{\ensuremath{L'(#1, #2)}}
\newcommand{\Wof}[2]{\ensuremath{W(#1, #2)}}
\newcommand{\Sof}[2]{\ensuremath{S(#1, #2)}}
\newcommand{\towers}{\ensuremath{\mathbb{T}}}
\newcommand{\tower}[1]{\ensuremath{T_{#1}}}
\newcommand{\fort}{\ensuremath{F}}
\newcommand{\twmax}{\ensuremath{W}}
\newcommand{\comp}{\ensuremath{\mathcal{C}}}

\newcommand{\G}{\ensuremath{\mathcal{B}}}
\newcommand{\Gplus}{\ensuremath{\G^{+}}}

\newcommand{\visibleboundary}[2]{\ensuremath{\partial_{vis(#1)}(#2)}}
\newcommand{\generatingset}{\ensuremath{\Gamma_{\G}}}
\newcommand{\cycleofedge}[1]{\ensuremath{\gamma(#1)}}

\newcommand{\boxshadow}[1]{\ensuremath{\sigma(#1)}}

\newcommand{\graphdist}[3]{\ensuremath{d_{#1}(#2, #3)}}

\newcommand{\connprob}{\ensuremath{p}}

\newcommand{\eps}{\varepsilon}

\newcommand{\N}{\ensuremath{\mathbb{N}}}

\newcommand{\Z}{\ensuremath{\mathbb{Z}}}

\DeclarePairedDelimiter{\Paren}{(}{)}
\DeclarePairedDelimiter{\Bracket}{[}{]}
\DeclarePairedDelimiter{\Brace}{\{}{\}}

\NewDocumentCommand{\Oh}{som}{%
  \ensuremath{\mathcal{O}\IfBooleanTF{#1}{%
    \Paren*{#3}%          % *-version => auto-size
  }{%
    \IfNoValueTF{#2}{%
      \Paren{#3}%         % no optional => normal size
    }{%
      \Paren[#2]{#3}%     % optional => \big, \Big, etc.
    }%
  }%
}}

\NewDocumentCommand{\Om}{som}{%
\ensuremath{\Omega\IfBooleanTF{#1}{%
    \Paren*{#3}%
  }{%
    \IfNoValueTF{#2}{%
      \Paren{#3}%
    }{%
      \Paren[#2]{#3}%
    }%
  }%
}}

\NewDocumentCommand{\Th}{som}{%
\ensuremath{\Theta\IfBooleanTF{#1}{%
    \Paren*{#3}%
  }{%
    \IfNoValueTF{#2}{%
      \Paren{#3}%
    }{%
      \Paren[#2]{#3}%
    }%
  }%
}}

\NewDocumentCommand{\oh}{som}{%
\ensuremath{o\IfBooleanTF{#1}{%
    \Paren*{#3}%
  }{%
    \IfNoValueTF{#2}{%
      \Paren{#3}%
    }{%
      \Paren[#2]{#3}%
    }%
  }%
}}

\NewDocumentCommand{\om}{som}{%
\ensuremath{\omega\IfBooleanTF{#1}{%
    \Paren*{#3}%
  }{%
    \IfNoValueTF{#2}{%
      \Paren{#3}%
    }{%
      \Paren[#2]{#3}%
    }%
  }%
}}

\NewDocumentCommand{\Prob}{som}{%
\ensuremath{\mathbb{P}% or \Pr, if you prefer
  \IfBooleanTF{#1}{%
    \Bracket*{#3}%
  }{%
    \IfNoValueTF{#2}{%
      \Bracket{#3}%
    }{%
      \Bracket[#2]{#3}%
    }%
  }%
}}

\NewDocumentCommand{\E}{som}{%
\ensuremath{\mathbb{E}
  \IfBooleanTF{#1}{%
    \Bracket*{#3}%
  }{%
    \IfNoValueTF{#2}{%
      \Bracket{#3}%
    }{%
      \Bracket[#2]{#3}%
    }%
  }%
}}

\NewDocumentCommand{\Max}{s o m m}{%
\ensuremath{\max
  \IfBooleanTF{#1}{%
    \Brace*{#3,\, #4}%
  }{%
    \IfNoValueTF{#2}{%
      \Brace{#3,\, #4}%
    }{%
      \Brace[#2]{#3,\, #4}
    }%
  }%
}}

\NewDocumentCommand{\Min}{s o m m}{%
\ensuremath{\min
  \IfBooleanTF{#1}{%
    \Brace*{#3,\, #4}%
  }{%
    \IfNoValueTF{#2}{%
      \Brace{#3,\, #4}%
    }{%
      \Brace[#2]{#3,\, #4}
    }%
  }%
}}

\NewDocumentCommand{\Exp}{som}{%
\ensuremath{\exp
  \IfBooleanTF{#1}{%
    \Paren*{#3}%
  }{%
    \IfNoValueTF{#2}{%
      \Paren{#3}%
    }{%
      \Paren[#2]{#3}%
    }%
  }%
}}

\newcommand{\calV}{\ensuremath{\mathcal{V}}}

\newcommand{\kl}[1]{\textcolor{magenta}{}}
\newcommand{\jl}[1]{\textcolor{cyan}{}}
\newcommand{\rr}[1]{\textcolor{orange}{}}

\begin{document}

\maketitle

\begin{abstract}
    We prove that the diameter of threshold (zero temperature) Geometric Inhomogeneous Random Graphs (GIRG) is $\Theta(\log n)$. This has strong implications for the runtime of many distributed protocols on those graphs, which often have runtimes bounded as a function of the diameter. 
    
    The GIRG model exhibits many properties empirically found in real-world networks, and the runtime of various practical algorithms has empirically been found to scale in the same way for GIRG and for real-world networks, in particular related to computing distances, diameter, clustering, cliques and chromatic numbers. Thus the GIRG model is a promising candidate for deriving insight about the performance of algorithms in real-world instances.

    The diameter was previously only known in the one-dimensional case, and the proof relied very heavily on dimension one. Our proof employs a similar Peierls-type argument alongside a novel renormalization scheme. Moreover, instead of using topological arguments (which become complicated in high dimensions) in establishing the connectivity of certain boundaries, we employ some comparatively recent and clearer graph-theoretic machinery. The lower bound is proven via a simple ad-hoc construction.
\end{abstract}

\newpage

\setcounter{page}{1}

\section{Introduction.}
The diameter of a graph is the maximum graph distance over all pairs of vertices. 
In case the graph is disconnected, we only consider pairs in the same connected component.
It has far-reaching implications on the performance of many distributed algorithms, where runtime bounds often depend on the diameter of the graph explicitly~\cite{peleg2000distributed}. The literature is too vast to do it justice here, so we just mention two examples, leader election~\cite{ghaffari2016leader} and minimum spanning tree~\cite{peleg2000distributed}, where the runtimes of state-of-the-art algorithms depend on the diameter. Thus it is important to understand the diameter of real-world networks and their models.

Many real-world networks commonly exhibit the following properties: power-law degree distributions~\cite{faloutsos1999power}, small-world distances (at most logarithmic) and large clustering coefficient\footnote{Consider choosing a random vertex and then two random neighbors of it. The probability that the latter vertices share an edge is the clustering coefficient.}~\cite{watts1998collective}, low-dimensionality~\cite{friedrich2023real} (assuming some underlying space for the vertices), hierarchical structure~\cite{clauset2008hierarchical}, navigability\footnote{As observed in the famous Milgram experiment~\cite{milgram1967small}, which became popularly known as ``six degrees of separation''.}~\cite{boguna2009navigability,bringmann2017greedy}, self-similarity~\cite{song2005self} and so on.
Motivated by this, numerous models for such networks have been introduced in the literature. One example, which we study in this work, is called \emph{Geometric Inhomogeneous Random Graphs (GIRG)}. This model generalizes Hyperbolic Random Graphs (HRG)~\cite{bringmann2019geometric}. Roughly speaking, vertices are randomly placed in Euclidean space (usually in a torus) and sample weights from a power law distribution, see~\Cref{def:simple-girg} below. Then, edge probabilities increase with the product of the weights and decrease with the distance. We focus on the \emph{threshold} case of this model (also called \emph{zero temperature GIRG} or \emph{T-GIRG}) in this work.

Our main motivation for studying the GIRG model stems from recent empirical works which show that the performance of various algorithms on GIRG closely matches that on real-world networks~\cite{blasius2024external,cerf2024balanced}. This includes in particular bidirectional breadth-first search for computing shortest paths, the iFUB algorithm~\cite{crescenzi2013computing} for the diameter, the dominance rule for vertex cover, the Louvain algorithm for clustering~\cite{blondel2008fast}, and the time required to enumerate all maximal cliques, and reduction rules for computing the chromatic number. %This indicates that the GIRG model is a good candidate for obtaining insights about practical algorithms, since the runtimes transfer to real-world networks in some important cases. 
While there are many other models for social and other complex real-world networks~\cite{leskovec2005graphs,leskovec2010kronecker,krioukov2010hyperbolic,deijfen2013scale,holland1981exponential,gracar2022recurrence}, we believe that these transfer results make the GIRG model a particularly promising candidate for gaining theoretical insights into the performance of algorithms, as they may sometimes translate into practical performances on real-world networks.

Moreover, there is compelling empirical evidence for the quality of fit between GIRG and certain real-world networks at least when properties like (degree) heterogeneity, clustering, typical distances and so on are the object of focus~\cite{blasius2024external, dayan2024expressivity, friedrich2023real}. For the special case of the HRG model, earlier works support the notion that real-world networks (for example a certain subset of the Internet) are amenable to embeddings in hyperbolic space~\cite{krioukov2010hyperbolic, garcia2019mercator}, and there is a mapping from such a space to 1-dimensional GIRG.\footnote{Roughly speaking, the angle of a vertex in the hyperbolic disk corresponds to its $x$ coordinate, and the distance from the center corresponds (inversely) to its weight.} 

% In addition, it is a natural model, evidenced by its independent introduction in the Theoretical Computer Science, Mathematics and Physics communities, under the names GIRG, Scale-Free percolation (SFP)~\cite{deijfen2013scale} and Hyperbolic Random Graphs (HRG)~\cite{krioukov2010hyperbolic}, respectively. Many properties of the model and processes on it have been studied extensively.

Let us now formally define the model. We first describe a power-law distribution, to be used for the weights of the vertices, which are in expectation in the order of their degrees.\footnote{There are more general versions of this definition, most notably incorporating so-called \emph{slowly-varying} functions as a scaling factor~\cite{voitalov2019scale}, which allows for better statistical fit to some real-world networks. Our proofs can be extended to such a scenario using standard tools (e.g.\ Potter bounds~\cite{bingham1989regular}), but we keep the current definition for simplicity of notation.} 
\begin{definition}\label{def:power-law}
    Let $\tau>1$. A continuous random variable $X\ge 1$ is said to follow a \emph{power-law with exponent} $\tau$ if it satisfies $\Prob{X \ge x} = \Th{x^{1-\tau}}$.
\end{definition}

We now define the precise graph model we will work with. 
For simplicity, throughout we use $\geomdist{u}{v}$ to denote the max norm distance between the positions of vertices $u$ and $v$.
Any other norm would do, as it would only change the distances by a constant factor at most, and our proofs are robust enough to handle such deviations. Also, throughout the paper we will assume that the dimension $d$, the density $\lambda$ and the power-law exponent $\tau$ are constant, while $n\to\infty$. Deviating slightly from the usual computer science convention, the expected number of vertices is $\lambda n$ instead of $n$, but this will not change any asymptotic results.
\begin{definition}\label{def:simple-girg}
Let $\tau>2, \density > 0$, $d,n\in\N$, and let $\mathcal{D}$ be a power-law distribution on $[1,\infty)$ with exponent $\tau$. A \emph{Threshold Geometric Inhomogeneous Random Graph (T-GIRG)} is sampled as follows:
    \begin{enumerate}
        \item The vertex set $\calV$ is given by a Poisson Point Process of intensity $\density$ on the d-dimensional torus $[0,n^{1/d}]^d$ of volume $n$.\footnote{Equivalently, we can sample the number of vertices from a Poisson distribution with expectation $\density n$ and draw their positions uniformly at random. This has mathematically slightly nicer properties than picking \emph{exactly} $\density n$ vertices, in particular, the number and types of vertices in any two disjoint regions of the space are independent. However, the differences are negligible, see~\cite[Claims~3.2,3.3]{komjathy2020explosion}.}
        \item Every vertex $u\in\mathcal{V}$ draws i.i.d.\ a \emph{weight} $\weightof{u} \sim \mathcal{D}$.

        \item For every two distinct vertices $u,v \in \mathcal{V}$, add an edge between $u$ and $v$ in $\mathcal{E}$ if and only if 
        \[
        \weightof{u} \weightof{v} \ge \volbetween{u}{v}.
        \]
    \end{enumerate}
\end{definition}
\begin{remark}
    This is only a special case of the GIRG model as defined in~\cite{bringmann2019geometric}. In the general case, the connection probability between two vertices scales as $(\weightof{u}\weightof{v}/\volbetween{u}{v})^{\alpha}$, where  the constant parameter $\alpha>1$ is the \emph{inverse temperature}. We use $\alpha = \infty$.
\end{remark}

\textbf{Previous works.} The diameter of GIRG has already been shown to be polylogarithmic, but with an unidentified exponent~\cite{bringmann2018averagedistancegeneralclass}. 
The difference between this bound and ours is still substantial for many modern distributed settings. 
Very closely related to this work is that of M{\"u}ller and Staps~\cite{muller2019diameter} which showed that the diameter of HRG is logarithmic.
As already noted, HRG can be thought of as GIRG with $d = 1$, and in fact some part of our upper bound proof reuses elements from the HRG case. However, the proof in~\cite{muller2019diameter}, in particular the use of blocking structures, relied heavily on dimension one. We also note that the resulting graphs for $d=1$ are \emph{structurally different}: for $d\ge 2$, the subgraph induced by vertices of degree at most $D$ contains a giant (linear-sized) component if $D$ is a sufficiently large constant. For $d=1$, the same induced subgraph is scattered and subcritical for every constant $D$.

\textbf{Organization.} The following sections (upper bound in~\Cref{sec:high_density,sec:any_density} and lower bound in~\Cref{sec:LB}) contain the derivation of our results, with proof sketches and intuition provided together with any claim. 
Some formal proofs are moved to the appendix due to space constraints. Altogether, we prove the following theorem. Traditionally, the most important regime is the one for $\tau\in(2,3)$, for which we obtain a diameter of $\Theta(\log n)$.
\begin{theoremEnd}[normal, text link=]{theorem}
    The following statements hold with high probability.\footnote{We say that an event holds \emph{with high probability (whp)} if it holds with probability $1 - o(1)$.}
    \begin{itemize}
        \item If $\tau = 3$ and $\density$ is a sufficiently large constant, then the diameter of T-GIRG is $\Oh{\log n}$.
        \item If $\tau < 3$, then the diameter of T-GIRG is $\Oh{\log n}$.
        \item If $\tau > 2$, then the diameter of T-GIRG is $\Om{\log n}$.
    \end{itemize}
\end{theoremEnd}
\begin{proofEnd}
    Upper bounds follow by~\Cref{lem:pathExistenceHighDensity,lem:pathcountinglemma} and~\Cref{cor:UB1,cor:UB2}. Lower bound follows by~\Cref{theorem:LB}.
\end{proofEnd}
We also provide in Appendix~\ref{sec:puzzle} some initial steps in extending the proof of the upper bound for a slight generalization of the model, in which we only keep each edge with probability $p$.
This happens in the general formulation of the GIRG model. 

\section{Upper bound: large enough \texorpdfstring{$\density$}{λ} and \texorpdfstring{$\tau \le 3$}{τ <= 3}.}\label{sec:high_density}
In this section we present the proof in a simplified setting, namely when the density $\density$ is large enough. This makes the proof considerably less technical and easier to follow. In the subsequent section we will then explain how this assumption can be removed (but only if $\tau < 3$). For the following, $\weightgeometry{} := [0,n^{1/d}]^d \times [1,\infty)$ is the geometric ground space augmented with the weight dimension. 

We follow some ideas from~\cite{muller2019diameter}. In particular, we tessellate $\weightgeometry$ in tree-structured boxes as follows (see~\Cref{fig:tessellation}). First, we fix a small constant $\Dzero$ such that $(n^{1/d})/\Dzero = 2^{\tessPower{0}}$, with $\tessPower{0}$ being an integer. The lowest level of boxes is then defined as:
\begin{displaymath}
\tessLevel{0} \;=\; 
\left\{
    \left(\prod_{k=1}^d 
  \left[
    (j_k - 1)\Dzero,\; j_k \Dzero
  \right)\right) \times \left[1, 2^{d/2}\right)
  \;\Big|\;
  1 \le j_k \le 2^{\tessPower{0}}
\right\}.
\end{displaymath}
In words, we tessellate\footnote{To have a proper partition of the ground space, we actually replace $\left[
    (j_k - 1)\Dzero,\; j_k \Dzero
  \right)$ by $\left[
    (j_k - 1)\Dzero,\; j_k \Dzero
  \right]$ whenever $j_k$ is the maximum permitted value. We only write this in this footnote for simplicity of the formulas.} the geometric space in the natural way by boxes of side length $\Dzero$ and we take for each the product with the weight range from $1$ to $2^{d/2}$. The choice of $2^{d/2}$ as the upper limit will be made clear soon.

We now similarly define the higher levels of the tessellation. The idea is that each box will have $2^d$ boxes of the next lowest level ``directly below'' it. That is, we define for $i < \tessPower{0}$:
\begin{displaymath}
\tessLevel{i} \;=\; 
\left\{
    \left(\prod_{k=1}^d 
  \left[
    2^i(j_k - 1)\Dzero,\; 2^i j_k \Dzero
  \right)\right) \times \left[(2^{d/2})^{i}, (2^{d/2})^{i + 1}\right)
  \;\Big|\;
  1 \le j_k \le 2^{\tessPower{0} - i}
\right\}.
\end{displaymath}

The highest level of the hierarchy consists of a single box covering the entire geometric ground space and the leftover weight range. That is, $\tessLevel{\tessPower{0}} = \groundspace \times \left[(2^{d/2})^{\tessPower{0}}, \infty\right)$.

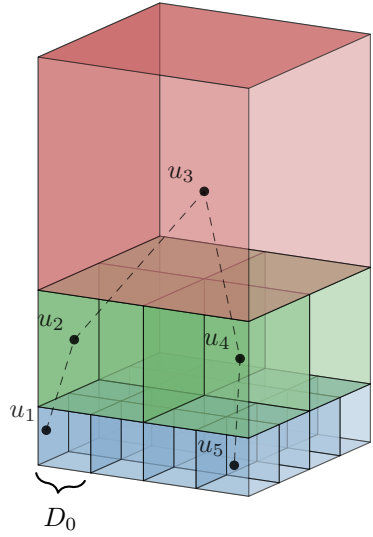
\begin{figure}
    \centering

% viewing angle
\tdplotsetmaincoords{75}{120}

\begin{tikzpicture}[tdplot_main_coords,
                    line join=round,line cap=round,scale=0.8]

%---------------- helper: draw one cube ------------------------------------
\newcommand{\drawcube}[5]{% #1 x  #2 y  #3 z  #4 size  #5 colour
  % back
  \filldraw[draw=black,line width=.4pt,fill=#5!30,opacity=.35]
    (#1,#2+#4,#3) -- ++(#4,0,0) -- ++(0,0,#4) -- ++(-#4,0,0) -- cycle;
  % left
  \filldraw[draw=black,line width=.4pt,fill=#5!45,opacity=.45]
    (#1,#2,#3) -- ++(0,#4,0) -- ++(0,0,#4) -- ++(0,-#4,0) -- cycle;
  % bottom
  \filldraw[draw=black,line width=.4pt,fill=#5!55,opacity=.5]
    (#1,#2,#3) -- ++(#4,0,0) -- ++(0,#4,0) -- ++(-#4,0,0) -- cycle;
  % right
  \filldraw[draw=black,line width=.4pt,fill=#5!60,opacity=.5]
    (#1+#4,#2,#3) -- ++(0,#4,0) -- ++(0,0,#4) -- ++(0,-#4,0) -- cycle;
  % front
  \filldraw[draw=black,line width=.4pt,fill=#5!70,opacity=.55]
    (#1,#2,#3) -- ++(#4,0,0) -- ++(0,0,#4) -- ++(-#4,0,0) -- cycle;
  % top
  \filldraw[draw=black,line width=.4pt,fill=#5!80,opacity=.6]
    (#1,#2,#3+#4) -- ++(#4,0,0) -- ++(0,#4,0) -- ++(-#4,0,0) -- cycle;
}

%---------------- colours for the three levels -----------------------------
\definecolor{lvl0}{RGB}{ 70,130,180}% steel‑blue
\definecolor{lvl1}{RGB}{ 34,139, 34}% forest‑green
\definecolor{lvl2}{RGB}{178, 34, 34}% firebrick‑red
%---------------- level‑0 cubes (side 1, z∈[0,1]) --------------------------
\foreach \x in {0,...,3}{
  \foreach \y in {0,...,3}{
    \drawcube{\x}{\y}{0}{1}{lvl0}
  }
}

% <<<‑‑‑ add brace right after the level‑0 cubes ‑‑‑>>>
% put the brace along the front‑bottom edge of the cube at (0,0,0)
\draw[decorate,
      decoration={brace,mirror,amplitude=6pt},
      thick]          % ← visual style
      (-2,-3.5,-2) --  % start a little below the edge
      ++(1, 1.5, 0.3)       % go one unit in x (the cube’s side length)
     node[midway,below=8pt] {$\Dzero$};
% <<<‑‑‑ end brace ‑‑‑>>>

%---------------- level‑1 cubes (side 2, z∈[1,3]) --------------------------
\foreach \x in {0,2}{
  \foreach \y in {0,2}{
    \drawcube{\x}{\y}{1}{2}{lvl1}
  }
}

%---------------- single level‑2 cube (side 4, z∈[3,7]) --------------------
\drawcube{0}{0}{3}{4}{lvl2}

% %---------------- dotted upward rays (level‑2 to ∞) ------------------------
% \foreach \c in {(0,0,7),(4,0,7),(4,4,7),(0,4,7)}{
%   \draw[dotted,very thick,->] \c -- ++(0,0,3);
% }

%--------------------- points & connecting poly‑line -----------------------
% coordinates
\coordinate (P1) at (2,-1,0);   % level 0
\coordinate (P2) at (1.6,-0.7,1.5);       % level 1
\coordinate (P3) at (2,2,4.5);       % level 2
\coordinate (P4) at (2.9,3.2,2);         % level 1
\coordinate (P5) at (2.4,2.8,0);   % level 0

% thick path
\draw[dashed, opacity=0.7] (P1) -- (P2) -- (P3) -- (P4) -- (P5);

% point markers + labels
\foreach \pt/\lbl in {P1/$u_1$,P2/$u_2$,P3/$u_3$,P4/$u_4$,P5/$u_5$}{
  \filldraw[black, opacity=0.85] (\pt) circle[radius=2pt]
         node[anchor=south east] {\lbl};
}

% (optional) axes
% \draw[->,thick] (0,0,0) -- (4.8,0,0) node[below] {$x$};
% \draw[->,thick] (0,0,0) -- (0,4.8,0) node[left]  {$y$};
% \draw[->,thick] (0,0,0) -- (0,0,8)  node[right] {$w$};

\end{tikzpicture}

\caption{Visualization of the tessellation for $d = 2$, with $3$ levels. The side length of lowest-level boxes is $\Dzero$. The blue boxes comprise $\tessLevel{0}$ (weight range $[1, 2)$), the green ones are $\tessLevel{1}$ (weight range $[2, 4)$) and the red box is $\tessLevel{2}$ (weight range $[4, \infty)$). Also depicted is a vertex path arising from a fully active canonical box path (as in~\Cref{def:canonboxpath}).} \label{fig:tessellation}
\end{figure}

Now, notice that any two adjacent boxes $\boxof{1}, \boxof{2}$ in this tessellation have the following property. For any vertices $x, y$ such that $x \in \boxof{1}$ and $y \in \boxof{2}$, we have $\weightof{x} \weightof{y} \ge \volbetween{x}{y}$ (for $\Dzero$ small enough).

This leads to the following observation. Consider any pair of vertices $u$ and $v$. From the box $\boxof{u}$ of $u$ there is a ``canonical path'' of boxes to the box $\boxof{v}$ of $v$, which is comprised of the two ``upward'' paths from $\boxof{u}$ and $\boxof{v}$ to their lowest common ancestor. If it were the case that all the boxes in this path contained at least one vertex, we would obtain a vertex-path (see~\Cref{fig:tessellation}) from $u$ to $v$ of length $\Oh*{\log{n}}$. To make this and the following more precise, we give the following definitions.

\begin{definition}
Consider the boxes of the tessellation of $\weightgeometry{}$ as a vertex set. On this vertex set we define $\G$ by adding an edge between boxes $B_1, B_2$ if their closures intersect in a $d$-dimensional set.\footnote{Mind that $\weightgeometry{}$ has dimension $d+1$.} We also define $\Gplus{}$ on the same vertex set by adding edges when the closures of said boxes have a non-empty intersection.
\end{definition}

Of course, the previously mentioned canonical path may fail to contain a vertex in each box. Nevertheless, we can utilize this idea to construct short paths between any two vertices $u$ and $v$, given that they are connected. To this end, let a box be called \emph{active} if it contains at least one vertex in the realization of the GIRG graph. For two boxes $\boxof{1}, \boxof{2}$, let $\Lof{\boxof{1}}{\boxof{2}}$ be the canonical box path (not necessarily active) connecting them. More formally:

\begin{definition}[Canonical box path]\label{def:canonboxpath}
    Let $\boxof{1}, \boxof{2}$ be two distinct boxes of the tessellation. For any box $B$ other than the top-most box $B_{top}$, let $\parentbox{B}$ be the unique box of minimum volume under the constraint that it is a strict superset of $B$ under projection along the weight dimension. Notice that $\{B, \parentbox{B}\} \in E(\G)$. Now, the set of edges identified in this way constitutes a spanning tree $T_{\G}$ of $\G$, rooted at $B_{top}$. We define the canonical box path $\Lof{B_1, B_2}$ from $B_1$ to $B_2$ as the unique path between them in $T_{\G}$. We may sometimes abuse notation and use $\Lof{u}{v}$ to mean $\Lof{\boxof{u}}{\boxof{v}}$.
\end{definition}

It turns out that we can bound the length of a shortest path between $u$ and $v$ by the cardinality of the set $\Wof{\boxof{u}}{\boxof{v}}$, defined below. Intuitively this set is constructed by taking $\Lof{\boxof{u}}{\boxof{v}}$ and ``expanding'' along inactive boxes. 

\begin{definition}[Inactive region of canonical box path]\label{def:inactiveregion}
    Let $\boxof{1}, \boxof{2}$ be two distinct boxes of the tessellation. Let $R$ be the subset of inactive boxes. Let $C_1^R, C_2^R, \ldots, C^R_c$ be the connected components of $\Gplus[R]$. We define $\Wof{\boxof{1}}{\boxof{2}}$ as 
    \begin{displaymath}
    \Wof{\boxof{1}}{\boxof{2}} := \Lprimeof{\boxof{1}}{\boxof{2}} \cup \left\{ \bigcup_{C_i^R \cap \Lprimeof{\boxof{1}}{\boxof{2}} \neq \emptyset}{C_i^R} \right\},
    \end{displaymath}
    where $\Lprimeof{\boxof{1}}{\boxof{2}} = \Lof{\boxof{1}}{\boxof{2}} \cup \{B \mid \text{$B$ is $\Gplus{}$-adjacent to some $B'$ in $\Lof{\boxof{1}}{\boxof{2}}$}\}$.
    We may sometimes abuse notation and use $\Wof{u}{v}$ to mean $\Wof{\boxof{u}}{\boxof{v}}$. Notice that $\Wof{\boxof{1}}{\boxof{2}}$ is a $\Gplus{}$-connected set of boxes. Additionally, any box $B \notin \Wof{\boxof{1}}{\boxof{2}}$ but $\Gplus{}$-adjacent to some box in $\Wof{\boxof{1}}{\boxof{2}}$ must be active. We use $\Sof{\boxof{1}}{\boxof{2}}$ to refer to the set of such boxes.
\end{definition}

To show a logarithmic upper bound on the diameter, it now suffices to show two statements. First, we will essentially show that if $u$ and $v$ are connected in the GIRG graph $G$, then $G^3$ contains a path from $u$ to $v$ where all vertices are in boxes of $\Wof{u}{v}\cup S(u,v)$.\footnote{$G^3$ is a graph where we add an edge between two vertices if and only if their distance is at most $3$ in $G$.} As a consequence, the length of a shortest path between any pair of vertices $u$ and $v$ is deterministically $\Oh*{|\Wof{u}{v}\cup S(u,v)|} = \Oh*{|\Wof{u}{v}|}$. Second, with high probability it holds that $|\Wof{u}{v}| = \Oh*{\log{n}}$ for all pairs $u, v$. Before proceeding to show these statements, we must lay some groundwork regarding $\G$ and $\Gplus$. This will be useful in establishing the first statement.

\subsection*{Boundary connectivity.}
Now we prove some useful structural properties about the graphs $\G$ and $\Gplus$. In particular, we will establish some preconditions of theorems proven in~\cite{timar2011bondaryconnectivitygraphtheory}. We will define the ``visible boundary'' of a set $C$ of boxes with respect to a given box $B$ as the set of $\Gplus{}$-neighbors of $C$ that are $\G{}$-reachable from $B$ without crossing $C$, see below. The goal is then to show that the visible boundaries of $\Gplus$-connected sets are $\G$-connected. Applying the result for $C:= \Wof{u}{v}$ will be useful later in the proof, where we construct shortcuts for subpaths that venture outside of $\Wof{u}{v} \cup \Sof{u}{v}$.

To illustrate the idea further and provide some intuition, let us consider for now such boundaries in the simplified case of $\Z^2$, where vertices are points of the infinite integer grid and nearest neighbors share an edge (in the counterpart of $\G$), as in~\Cref{fig:2DBoundaries}. For the analogue of $\Gplus$, which we will call $\Z^{2*}$, we also add the edges between vertices whose $x$ and $y$ coordinates both differ by 1. Each vertex of this infinite graph should intuitively be thought of as a box in our tessellation, even though our tessellation does not possess the edge set of $\Z^2$ or $\Z^{2*}$. Assume now that we are presented with a blue/red coloring of the vertices of $\Z^2$, where blue is generally meant to signify an active box and red generally signifies an inactive box. Given such a coloring with the origin being red, we can define the ``red component'' $C_{red}$ (comparable to $\Wof{u}{v}$ for our purposes) of the origin as its connected component in $\Z^{2*}$ induced by red vertices. This may possibly yield a red region without ``holes'', surrounded by a connected (in $\Z^2$) blue boundary (as in~\Cref{fig:2DBoundaries}, left). However, this is not the only case. For example, $C_{red}$ could resemble a donut, in which case the blue boundary consists of two connected circles. Now, the points inside the donut hole ``see'' a connected part of the boundary and so do the points outside the donut. This is a crucial property which we will use in the following way. Thinking of the red region as $\Wof{u}{v}$ and of the blue boundary as $\Sof{u}{v}$, consider the graph path connecting $u$ to $v$. If the path stays within the union of the red region and the blue boundary, it is already short enough and there is nothing we need to do. If, however, this walk ventures outside of $\Wof{u}{v} \cup \Sof{u}{v}$, we will show that there is a shortcut walking along the visible boundaries to obtain a new walk that avoids such excursions, except perhaps for the first and last vertices of each. More exotic cases can arise when the path successively jumps through a sequence of different holes, but these can also be tamed as we will see.

% Note that the path may contain edges $\{x,y\}$ that jump across the blue region. But as we will show later, this deterministically implies that at least one of $x$ and $y$ also forms an edge with a vertex in the blue region, meaning that from both $x$ and $y$ we can reach the blue region with a path of length $1$ or $2$.

\begin{figure}[ht]
\centering
\begin{tikzpicture}[scale=0.4, line join=round, line cap=round]

% ---------- Helper styles ----------
\tikzset{
  gridpt/.style={fill=black!30, draw=none},
  comppt/.style={fill=red, draw=red},        % component (red)
  boundpt/.style={fill=blue, draw=blue},     % boundary (blue)
  compcell/.style={fill=red, fill opacity=0.3, draw=none},
  boundcell/.style={fill=blue, fill opacity=0.3, draw=none},
  labelstyle/.style={font=\small},
  pathseg/.style={draw=orange!85!black, very thick},
  shortcut/.style={draw=purple, dashed, ultra thick, dash pattern=on 6pt off 4pt},
  vertex/.style={fill=black, draw=white, line width=0.3pt, circle, minimum size=6pt, inner sep=0pt}
}

\newcommand{\cell}[3]{% style, x, y
  \path[#1] ($(#2-0.5,#3-0.5)$) rectangle ++(1,1);
}

% ===============================================================
% LEFT PANEL: simply connected red component with one blue loop
% ===============================================================
\begin{scope}[shift={(-7,0)}]
  \foreach \x in {-5,...,5}
    \foreach \y in {-5,...,5}
      \fill[gridpt] (\x,\y) circle (0.6pt);

  % Red block (component)
  \foreach \x in {-2,...,2}
    \foreach \y in {-2,...,2}{
      \cell{compcell}{\x}{\y}
      \fill[comppt] (\x,\y) circle (1.8pt);
    }

  % Blue boundary
  \foreach \x in {-3,...,3}{
    \cell{boundcell}{\x}{3} \fill[boundpt] (\x,3) circle (1.8pt);
    \cell{boundcell}{\x}{-3} \fill[boundpt] (\x,-3) circle (1.8pt);
  }
  \foreach \y in {-2,...,2}{
    \cell{boundcell}{3}{\y} \fill[boundpt] (3,\y) circle (1.8pt);
    \cell{boundcell}{-3}{\y} \fill[boundpt] (-3,\y) circle (1.8pt);
  }

  % Start/end vertices in red (not on grid)
  \coordinate (Ls) at (-1.2,-1.15);
  \coordinate (Le) at ( 1.25, 1.1);
  \node[vertex] at (Ls) {};
  \node[labelstyle] at ($(Ls)+(-0.2,-0.5)$) {$v_{\mathrm{start}}$};
  \node[vertex] at (Le) {};
  \node[labelstyle] at ($(Le)+( 0.45, -0.4)$) {$v_{\mathrm{end}}$};

  \coordinate (orig) at (0, 0);
  \node[vertex] at (orig) {};
  \node[labelstyle] at ($(orig)+(0.2,0.5)$) {$\mathcal{O}$};

  % Outside excursion near the top; return before ending
  \coordinate (LoutA) at (-0.4,  3.20); % cross top blue to outside
  \coordinate (LoutB) at ( 1.0,  3.15); % re-enter across top blue

  % Main path: start -> exit (top) -> wander outside -> return -> end
  \draw[pathseg, rounded corners=12pt]
    (Ls)
      .. controls (-1.4,0.6) and (-0.9,2.2) .. (LoutA)
      .. controls (-0.2,4.6) and ( 1.2,4.4) .. (LoutB)
      .. controls ( 1.3,2.5) and ( 1.6,1.8) .. (Le);

  % Shortcut entirely inside blue band (y≈3)
  \draw[shortcut, rounded corners=8pt]
    ($(LoutA)+(0,0.05)$)
      .. controls (-0.2,3.05) and (0.5,3.05) .. ($(LoutB)+(0,0.05)$);

  \node[labelstyle] at (0,-5.5) {Simply connected \(C_{\text{red}}\)};
\end{scope}

% ===============================================================
% RIGHT PANEL: donut-shaped red component with two blue loops
%   1) large outside excursion near the bottom (return completed there)
%   2) separate excursion into the hole
%   Shortcuts lie inside the corresponding blue bands
% ===============================================================
\begin{scope}[shift={(7,0)}]
  \foreach \x in {-6,...,6}
    \foreach \y in {-6,...,6}
      \fill[gridpt] (\x,\y) circle (0.6pt);

  % Red ring (component: between "radius" 2 and 4)
  \foreach \x in {-4,...,4}{
    \cell{compcell}{\x}{4} \fill[comppt] (\x,4) circle (1.8pt);
    \cell{compcell}{\x}{-4} \fill[comppt] (\x,-4) circle (1.8pt);
  }
  \foreach \y in {-3,...,3}{
    \cell{compcell}{4}{\y} \fill[comppt] (4,\y) circle (1.8pt);
    \cell{compcell}{-4}{\y} \fill[comppt] (-4,\y) circle (1.8pt);
  }

  % Outer blue boundary
  \foreach \x in {-5,...,5}{
    \cell{boundcell}{\x}{5} \fill[boundpt] (\x,5) circle (1.8pt);
    \cell{boundcell}{\x}{-5} \fill[boundpt] (\x,-5) circle (1.8pt);
  }
  \foreach \y in {-4,...,4}{
    \cell{boundcell}{5}{\y} \fill[boundpt] (5,\y) circle (1.8pt);
    \cell{boundcell}{-5}{\y} \fill[boundpt] (-5,\y) circle (1.8pt);
  }

  % Inner blue boundary (around the hole)
  \foreach \x in {-3,...,3}{
    \cell{boundcell}{\x}{3} \fill[boundpt] (\x,3) circle (1.8pt);
    \cell{boundcell}{\x}{-3} \fill[boundpt] (\x,-3) circle (1.8pt);
  }
  \foreach \y in {-2,...,2}{
    \cell{boundcell}{3}{\y} \fill[boundpt] (3,\y) circle (1.8pt);
    \cell{boundcell}{-3}{\y} \fill[boundpt] (-3,\y) circle (1.8pt);
  }

  % Start/end vertices on the red ring (not grid points)
  \coordinate (Rs) at ( 3.55, 0.9);
  \coordinate (Re) at (-3.55,-0.9);
  \node[vertex] at (Rs) {};
  \node[labelstyle] at ($(Rs)+( 0.6, 0.45)$) {$v_{\mathrm{start}}$};
  \node[vertex] at (Re) {};
  \node[labelstyle] at ($(Re)+(-0.4,-0.45)$) {$v_{\mathrm{end}}$};

  \coordinate (origR) at (0, 4);
  \node[vertex] at (origR) {};
  \node[labelstyle] at ($(origR)+(0.5,0.2)$) {$\mathcal{O}$};

  % --- (1) Outside excursion near the bottom (across outer blue band) ---
  \coordinate (BoutA) at ( 0.8,-5.25);  % exit to outside
  \coordinate (BoutB) at (-1.0,-5.25);  % re-enter from outside

  % --- (2) Separate excursion into the hole (across inner blue band) ---
  \coordinate (Hin)  at ( 2.6, 0.25);   % enter hole
  \coordinate (Hout) at ( 0.25, 1.0);   % exit hole back to annulus

  % Main path:
  % start on red -> go down to outside -> come back -> then into hole -> back -> end
  \draw[pathseg, rounded corners=12pt]
    (Rs)
      .. controls ( 3.2, 0.1) and ( 4.0,-2.6) .. ( 1.4,-4.1)
      .. controls ( 1.2,-4.6) and ( 1.0,-4.9) .. (BoutA)
      .. controls ( 0.6,-6.1) and (-0.6,-6.2) .. (BoutB)
      .. controls (-1.2,-4.8) and (-2.4,-3.2) .. (-1.8,-2.4)
      .. controls (-2.2,-1.0) and ( 0.6,-0.2) .. (Hin)
      .. controls ( 1.2, 0.3) and ( 0.6, 1.2) .. (Hout)
      .. controls (-0.8, 2.8) and (-2.4, 1.0) .. (Re);

  % Shortcuts confined to the blue bands:
  %   (a) along the outer blue band (y≈-5)
  \draw[shortcut, rounded corners=8pt]
    ($(BoutA)+(0,0.15)$)
      .. controls ( 0.2,-5.05) and (-0.4,-5.05) .. ($(BoutB)+(0,0.15)$);

  %   (b) along the inner blue band (around radius 3)
  \draw[shortcut, rounded corners=8pt]
    (-2.8, -0.05)
      .. controls ( -3, -1.9) and ( -3.2, -2) .. (-2.4, -3);

  \node[labelstyle] at (0,-6.9) {Donut-shaped \(C_{\text{red}}\)};
\end{scope}

\end{tikzpicture}
\caption{Two examples of a red connected component \(C_{\text{red}}\) and its blue boundary with paths that start and end in \(C_{\text{red}}\).
Left: one outward excursion; the shortcut stays inside the connected blue boundary.
Right: first an outward excursion near the bottom, then a separate excursion into the hole; each has a shortcut confined to the corresponding connected part of the blue boundary.}\label{fig:2DBoundaries}
\end{figure}
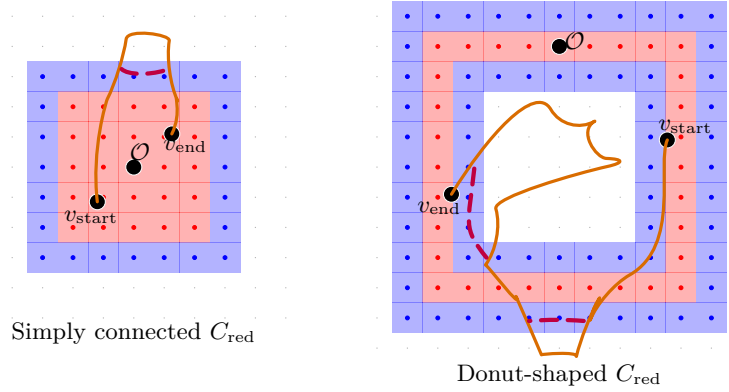

With the previous example guiding our intuition and keeping in mind that there are new subtleties to uncover, let us now define visible boundaries in accordance with~\cite{timar2011bondaryconnectivitygraphtheory}. As we have mentioned, we will use these boundaries to ``patch'' parts of our constructed path, using their connectivity, which we will prove later.
\begin{definition}
    For a $\Gplus$-connected subset of $C$ and a box $B$, we define the \emph{boundary of $C$ visible from $B$} as
    \begin{displaymath}
        \visibleboundary{B}{C} = \{B' \mid \text{$B'$ is $\Gplus$-adjacent to some $B'' \in C$ and $B'$ is connected to $B$ in $\G \setminus C$}\}
    \end{displaymath}
\end{definition}

Notice that we require $B'$ to be $\Gplus$-adjacent to $C$, but we only allow paths in $\G$ for the connection to $B$. This is required for the theorems in~\cite{timar2011bondaryconnectivitygraphtheory} to apply, but otherwise plays little role for our analysis.

Before introducing the theorem we will use, let us describe the preconditions we need to prove. The first one is with regards to the \emph{cycle space} of $\G$. The cycle space is simply the set of subgraphs of $G$ where each vertex has an even degree. Our goal is to find a suitable generating set $\generatingset{}$ for this space, with symmetric difference of edges being the underlying algebraic operation. For this first criterion, we regard a generating set as \emph{suitable} if all $C \in \generatingset{}$ are chordal in $\Gplus$. We say that a cycle is \emph{chordal} in $\Gplus$ if its vertices induce a clique in $\Gplus$. For reference, such a set for the pair $\Z^2, \Z^{2*}$ is the set of all 4-step cycles.

It is a standard fact that so-called ``fundamental cycles'' generate the cycle space of a graph~\cite[Theorem~1.9.5]{diestel2017}. To produce a set of fundamental cycles, pick a spanning forest $F$. Then, for any $e \in G \setminus F$, $F \cup \{e\}$ has a unique cycle. The set of cycles found like so is referred to as a set of fundamental cycles.

It therefore suffices to find a generating set for a set of fundamental cycles. To do so, consider the natural spanning tree $T_{\G}$ of $\G$, comprised of all edges of $\G$ connecting boxes of different levels (see~\Cref{fig:SpanningTreeFundCycle}). We will observe that fundamental cycles of this tree follow a very restricted structure which will guide us into defining $\generatingset{}$.
\begin{observation}\label{obs:fundamentalcycles}
    All fundamental cycles of $T_{\G}$ can be expressed as the union of $e, P_{B_1}$ and $P_{B_2}$, where $e = \{B_1, B_2\}$ is the added edge and $P_{B_i}$ is the unique path in $T_{\G}$ from $B_i$ to the lowest common ancestor of $B_1, B_2$, assuming $T_{\G}$ is rooted at the highest level box. Moreover, for each $B'_1$ and $B'_2$ in the same level of the tessellation and with $B'_{i}$ internal in $P_{i}$, we have $\{B'_1, B'_2\} \in E(\G)$.
\end{observation}
\begin{proof}
    The first part of the observation follows from the fact that the described union constitutes a cycle and by uniqueness of said cycle. For the second part, notice that for any two boxes in the same level of the tessellation with  $\{B'_1, B'_2\} \notin E(\G)$, we know that the projections of the boxes along the weight dimension do not intersect, even after taking their closures. Since all descendants of these boxes have projections that are subsets of the above, no pair of them intersects either, and thus the claim follows by contraposition.
\end{proof}

We are now in a position to define our generating set.

\begin{definition}
    Let $e = \{B_1, B_2\} \in E(\G)$ be an edge between two boxes of the same level. Let $B'_i = \parentbox{B_i}$ be the parent box of $B_i$. If $B'_1 = B'_2 = B'$, then let $\cycleofedge{e} = (B_1, B', B_2, B_1)$, otherwise let $\cycleofedge{e} = (B_1, B'_1, B'_2, B_2, B_1)$. We define $\generatingset{}$ as the set containing all such cycles.
\end{definition}

\begin{theoremEnd}[normal, text link=]{lemma}
    The set $\generatingset{}$ generates the cycle space of $\G$. Moreover, each element of $\generatingset{}$ is chordal in $\Gplus$.
\end{theoremEnd}
\begin{proofEnd}
    Recall that it suffices to show that $\generatingset{}$ generates all fundamental cycles of $T_{\G}$. Let $C$ be such a cycle where (see also~\Cref{fig:SpanningTreeFundCycle}) 
    \begin{displaymath}
        C = (B_{high}, B_1^L, B_2^L, \ldots, B_k^L, B_k^R, B_{k-1}^R, \ldots, B_1^R,B_{high})
    \end{displaymath}
    with $\{B_k^L, B_k^R\}$ being the added edge. Recall that by~\Cref{obs:fundamentalcycles} we have $\{B_i^L, B_i^R\} \in E(\G)$, meaning $\cycleofedge{\{B_i^L, B_i^R\}} \in \generatingset{}$ for all $i$. Consider $S = \sum_{i}\cycleofedge{\{B_i^L, B_i^R\}}$. Notice that all edges of $C$ appear exactly once in this sum. Additionally, all other edges ($\{B_j^L, B_j^R\}$ for $j < k$) in the sum appear exactly twice. Therefore, $S = C$, yielding the first claim.

    We will now show the chordality in $\Gplus$ of elements of $\generatingset{}$. Consider an element $\cycleofedge{e} \in \generatingset{}$. Because the edge $e$ is incident to boxes of the same level, we conclude that these boxes can be characterized by vectors $J_1$ and $J_2$ as in the definition of the tessellation, where $J_1$ and $J_2$ differ by 1 in exactly one of $d$ coordinates. Therefore, their parents also satisfy this property, or coincide. In either case, one can identify a $(d - 1)$-dimensional set of points in $d + 1$ dimensions, defined by restricting the weight to be the lower limit of the parents and the coordinate in which $J_1$ and $J_2$ differ to be equal to the value where the closures of the boxes of $e$ intersect. All other coordinates are free to take any value in the intervals prescribed by the boxes of $e$. This set of points guarantees that the considered cycle is chordal in $\Gplus$. 
\end{proofEnd}

With the condition on $\generatingset{}$ satisfied for $\G{}$ and $\Gplus{}$, let us present the theorem we want to use, which is an adapted version of Theorem 4 in~\cite{timar2011bondaryconnectivitygraphtheory}.
\begin{theorem}\label{theorem:boundary_connectivity}
Let $\Gplus{}$ be a connected graph, and $\G{}$ a connected subgraph of $\Gplus{}$. Suppose that there is a generating set $\generatingset{}$ for the cycle space of $\G{}$ that is chordal in $\Gplus{}$, and that for every edge $e\in \Gplus{}$ there is a cycle $O_e$ (containing $e$) in $\Gplus{}$ such that $O_e\setminus e\subset \G{}$, and $O_e$ is chordal in $\Gplus{}$. Let $C$ be a connected subgraph of $\Gplus{}$, and $x\in V(\G{})\setminus C$. Then $\visibleboundary{x}{C}$ is connected in $\G{}$.
\end{theorem}

We notice the extra condition about the existence of the cycles $O_e$. Let us now record this condition for our particular pair $\G, \Gplus$, in a way which will be useful later on as well.

\begin{theoremEnd}[restate, text link=]{claim}\label{claim:shortcutGplus}
    Let $B_1, B_2$ be two boxes connected by an edge $e$ in $\Gplus{}$. Consider the graph $G_{local}$ with vertex set $V = \{B \mid \graphdist{\Gplus{}}{B}{B_1} \le 1\}$ and edge set $E(\G) \cap V^2 \setminus \{e\}$. Then, $G_{local}$ contains a path from $B_1$ to $B_2$. Additionally, this path along with $e$ constitutes a cycle $O_e$ as prescribed by~\Cref{theorem:boundary_connectivity}.
\end{theoremEnd}
    \begin{proofEnd}
        Let us first consider the case where $B_1$ and $B_2$ are boxes of the same level. In this case, these boxes correspond to vectors $J_1, J_2$ which differ in any coordinate by at most $1$. Changing the disagreeing coordinates one by one reveals the desired path in $G_{local}$.
        
        The only other case is when $B_1$ and $B_2$ are in subsequent levels. Now, assume wlog that $B_2$ is of the higher level, and consider its ``children'' boxes, which are in the same level as $B_1$. If $B_1$ is included in the children, the claim follows easily. Otherwise, it must hold that $B_1$ is at least $\Gplus{}$-adjacent to one of the children. This follows because $B_2$'s children cover its projection along the weight dimension, which $B_1$'s closure must intersect. Now we are again essentially in the first case, and we can change the disagreeing coordinates one by one until we reach a child of $B_2$. We then append the edge from said child to $B_2$. It is easy to see that we never leave $V$, since in the first same-level part no coordinate-wise difference is pushed beyond 1 and $B_2$ is in $V$ by assumption. For the chordality of the cycle, we can use a similar argument along with the fact that subsequent weight levels have intersecting weight interval closures.
    \end{proofEnd}

With~\Cref{theorem:boundary_connectivity} in our arsenal, we may proceed to relate $|\Wof{u}{v}|$ to the path length.

\begin{figure}[!ht]
\tdplotsetmaincoords{70}{-30}

    \centering
    \begin{minipage}[t]{0.49\textwidth}
        \centering
         \begin{tikzpicture}[tdplot_main_coords,
                        line join=round,line cap=round,scale=0.7]
    
    % helper macro for cubes
    \newcommand{\drawcube}[5]{% #1 x  #2 y  #3 z  #4 size  #5 colour
      \filldraw[draw=black,line width=.3pt,fill=#5!40,opacity=.4]
        (#1,#2,#3) -- ++(#4,0,0) -- ++(0,#4,0) -- ++(-#4,0,0) -- cycle;
      \filldraw[draw=black,line width=.3pt,fill=#5!40,opacity=.4]
        (#1,#2,#3) -- ++(0,#4,0) -- ++(0,0,#4) -- ++(0,-#4,0) -- cycle;
      \filldraw[draw=black,line width=.3pt,fill=#5!40,opacity=.4]
        (#1,#2,#3) -- ++(#4,0,0) -- ++(0,0,#4) -- ++(-#4,0,0) -- cycle;
       \filldraw[draw=black,line width=.3pt,fill=#5!40,opacity=.4]
        (#1+#4,#2+#4,#3+#4) -- ++(-#4,0,0) -- ++(0,0,-#4) -- ++(#4,0,0) -- cycle;
        \filldraw[draw=black,line width=.3pt,fill=#5!40,opacity=.4]
        (#1+#4,#2+#4,#3+#4) -- ++(0,-#4,0) -- ++(0,0,-#4) -- ++(0,#4,0) -- cycle;
        \filldraw[draw=black,line width=.3pt,fill=#5!40,opacity=.4]
        (#1+#4,#2+#4,#3+#4) -- ++(-#4,0,0) -- ++(0,-#4,0) -- ++(#4,0,0) -- cycle;
    }
    
    % colours
    \definecolor{lvl0}{RGB}{ 70,130,180}
    \definecolor{lvl1}{RGB}{ 34,139, 34}
    \definecolor{lvl2}{RGB}{178, 34, 34}
    
    % draw a small tessellation
    \drawcube{0}{0}{3}{4}{lvl2}
    \drawcube{0}{0}{1}{2}{lvl1}
    \drawcube{2}{0}{1}{2}{lvl1}
    \foreach \x in {0,1,2,3}{\drawcube{\x}{0}{0}{1}{lvl0}}
    
    % coordinates
    \coordinate (R) at (2.6,2,5);
    \coordinate (G1) at (1.2,1,1.8);
    \coordinate (G2) at (3,1,1.8);
    \coordinate (B1) at (0.5,0.5,0);
    \coordinate (B2) at (1.5,0.5,0);
    \coordinate (B3) at (2.5,0.5,0);
    \coordinate (B4) at (3.5,0.5,0);
    
    % spanning tree
    \draw[thick] (R)--(G1) (R)--(G2);
    \draw[thick] (G1)--(B1) (G1)--(B2);
    \draw[thick] (G2)--(B3) (G2)--(B4);
    
    % % added edge
    % \draw[ultra thick,red,dashed] (B2)--(B3);
    
    % fundamental cycle (blue)
    \draw[ultra thick,blue,opacity=0.7] (B2)--(G1)--(R)--(G2)--(B3)--cycle;
    
    % generating cycle 1 (triangle: green)
    \draw[thick,green!70!black,dotted] (G1)--(R)--(G2)--cycle;
    
    % generating cycle 2 (quad: orange)
    \draw[thick,orange!90!black,dotted] (B2)--(G1)--(G2)--(B3)--cycle;
    
    % nodes
    \foreach \pt/\lbl in {R/$B_{high}$,G1/$B_1^L$,G2/$B_1^R$,B1/,B4/}{
      \filldraw[black] (\pt) circle (2.5pt) node[above right=0pt] {\lbl};
    }

    \foreach \pt/\lbl in {B2/$B_2^L$,B3/$B_2^R$}{
      \filldraw[black] (\pt) circle (2.5pt) node[below right=1pt] {\lbl};
    }

    \end{tikzpicture}
        \caption{Part of the spanning tree $T_{\G}$. The fundamental cycle (blue) formed by adding the edge $\{B_2^L, B_2^R\}$ is the symmetric difference of the (green) triangle $(B_1^L,B_{high},B_1^R)$ and the (orange) quadrilateral $(B_2^L,B_1^L,B_1^R,B_2^R)$, both in $\generatingset{}$.}
        \label{fig:SpanningTreeFundCycle}
    \end{minipage}
    \hfill
    \begin{minipage}[t]{0.49\textwidth}
        \centering
        \tdplotsetmaincoords{70}{120}

        \begin{tikzpicture}[tdplot_main_coords, line join=round, line cap=round, >=Latex, scale=0.8]
        
        % --- parameters
        \def\s{3} % cube side
        
        % --- choose endpoints so z(A)<0 and z(B)>\s and the segment pierces the bottom/top faces inside the square
        \pgfmathsetmacro{\Ax}{-1}
        \pgfmathsetmacro{\Ay}{ 2.5}
        \pgfmathsetmacro{\Az}{-1.0}
        \pgfmathsetmacro{\Bx}{ 1.9}
        \pgfmathsetmacro{\By}{ -.8}
        \pgfmathsetmacro{\Bz}{ 4.0}
        
        \coordinate (A) at (\Ax,\Ay,\Az);
        \coordinate (B) at (\Bx,\By,\Bz);
        
        % --- intersections with z=0 and z=\s (inside the square by construction)
        \pgfmathsetmacro{\tbot}{(0 - \Az)/(\Bz - \Az)}
        \pgfmathsetmacro{\xtbot}{\Ax + \tbot*(\Bx-\Ax)}
        \pgfmathsetmacro{\ytbot}{\Ay + \tbot*(\By-\Ay)}
        \coordinate (Ibot) at (\xtbot,\ytbot,0);
        
        \pgfmathsetmacro{\ttop}{(\s - \Az)/(\Bz - \Az)}
        \pgfmathsetmacro{\xttop}{\Ax + \ttop*(\Bx-\Ax)}
        \pgfmathsetmacro{\yttop}{\Ay + \ttop*(\By-\Ay)}
        \coordinate (Itop) at (\xttop,\yttop,\s);
        
        % --- semi-transparent faces (so hidden dashed bit is still faintly visible through)
        % \fill[gray!30,opacity=0.25] (0,0,\s) -- (\s,0,\s) -- (\s,\s,\s) -- (0,\s,\s) -- cycle; % top
        % \fill[gray!30,opacity=0.20] (\s,0,0) -- (\s,\s,0) -- (\s,\s,\s) -- (\s,0,\s) -- cycle; % right
        % \fill[gray!30,opacity=0.15] (0,\s,0) -- (\s,\s,0) -- (\s,\s,\s) -- (0,\s,\s) -- cycle; % back

        % --- segment: solid outside cube, dashed inside cube (occluded)
        \draw[very thick,blue] (A) -- (Ibot);
        \draw[very thick,blue,dashed] (Ibot) -- (Itop);
        \draw[very thick,blue] (Itop) -- (B);

        % vertex in box
        \coordinate (Z) at (1.5, 0.4, 0);
        \draw[very thick, dotted, magenta] (Z) -- (B);
        \fill[magenta] (Z) circle (1.4pt) node[below left=1pt] {$z$};
        % endpoints
        \fill[blue] (A) circle (1.4pt) node[above right=1pt] {$x$};
        \fill[blue] (B) circle (1.4pt) node[above right=1pt] {$y$};

        % --- cube edges
        \draw[black, thick]
          (0,0,0) -- (\s,0,0) -- (\s,\s,0) -- (0,\s,0) -- cycle
          (0,0,\s) -- (\s,0,\s) -- (\s,\s,\s) -- (0,\s,\s) -- cycle
          (0,0,0) -- (0,0,\s)
          (\s,0,0) -- (\s,0,\s)
          (\s,\s,0) -- (\s,\s,\s)
          (0,\s,0) -- (0,\s,\s);

        \end{tikzpicture}
        \caption{An active box intersected by an edge $\{x, y\}$. The vertex $z$ inside the box is guaranteed to connect to one of the endpoints (by~\Cref{lem:boxcrossinglemma}), in this case $y$. Note that the ``vertical'' coordinate here signifies weight.} \label{fig:intersected_box}
    \end{minipage}
\end{figure}

\subsection*{Bounding the shortest path length by $|\Wof{u}{v}|$.}
As mentioned earlier, our approach will be to shortcut parts of (assumed) shortest paths which stray meaningfully out of $\Wof{u}{v}$. The following lemma is paramount for this goal. Intuitively, it allows us to ``anchor'' the path to the boundary of $\Wof{u}{v}$ whenever it is crossed. See~\Cref{fig:intersected_box} for a visualization.

\begin{theoremEnd}[restate, text link=]{lemma}\label{lem:boxcrossinglemma}
    Let $\Dzero$ in the definition of the tessellation be small enough. Let $x, y$ be two vertices connected by an edge intersecting an active box $B$. Then, $B$ contains a vertex $z$ that is connected to either $x$ or $y$ by an edge.
\end{theoremEnd}

\newcommand{\epsForWeightZoverX}{\ensuremath{\eps_1}}
\newcommand{\epsForWeightXoverZ}{\ensuremath{\eps_1'}}
\newcommand{\CForXYDistOverDi}{\ensuremath{c_1}}
\newcommand{\dmin}{\ensuremath{s_{min}}}
\newcommand{\DeltaWMax}{\ensuremath{\Delta_w}}
\newcommand{\epsForSlope}{\ensuremath{\eps_2}}
\newcommand{\CForXYDistOverWZsquared}{\ensuremath{c_2}}
\newcommand{\logC}{\ensuremath{\xi}}
\newcommand{\rhoX}{\ensuremath{\rho_x}}
\newcommand{\slopeVal}{\ensuremath{\sigma}}

\begin{proofEnd}
    In the following, we assume both $x$ and $y$ (as points in $\weightgeometry{}$) are outside of $B$, otherwise the claim easily follows.
    Let $z$ be one of the vertices $B$ is guaranteed to contain due to being active. Assume $B$ is a box of weight level $i$. Thus, $\weightof{z} \ge 2^{di/2}$ and let us assume equality wlog. Also assume wlog that $\weightof{y} \ge \weightof{x}$ and notice that this implies $\weightof{y} \ge \weightof{z}$ and $\weightof{x} \le 2^{d/2} \weightof{z}$ (otherwise the edge would not intersect the box).
        
     First of all, if $\geomdist{x}{y} \le \CForXYDistOverDi \Di{i}$ then $\volbetween{y}{z} \le (\geomdist{x}{y} + \Di{i})^d \le (1 + \CForXYDistOverDi)^d\Di{i}^d$. So, since $\weightof{y} \weightof{z} \ge 2^{di} = (1/\Dzero)^d \Di{i}^d$, $z$ would connect to $y$ if $1 + \CForXYDistOverDi \le 1/\Dzero$. So, assume from now on that $\geomdist{x}{y} \ge \Di{i}(1-\Dzero)/\Dzero$, which in turn implies that $\geomdist{y}{z}$ and $\geomdist{x}{z}$ are both at most $(1 + \Dzero/(1-\Dzero))\geomdist{x}{y}$.

    \paragraph*{Small $\weightof{x}$.} Consider for a moment the case $\weightof{z} \ge (1 + \epsForWeightZoverX) \weightof{x}$ for $\epsForWeightZoverX$ a function of $\Dzero$ to be specified. By the previous restriction on $\geomdist{y}{z}$, we see that $z$ would connect to $y$ if additionally $1 + \epsForWeightZoverX \ge (1 + \Dzero/(1-\Dzero))^d$.

    \paragraph*{Large $\weightof{x}$.} So, we can assume for the rest that $\weightof{z} \le (1 + \epsForWeightZoverX) \weightof{x}$, with $\epsForWeightZoverX = (1 + \Dzero/(1-\Dzero))^d - 1$. Notice that for $\Dzero$ going to $0$, this expression also goes to $0$. For convenience, let $\epsForWeightXoverZ$ be such that $\weightof{x} \ge (1 - \epsForWeightXoverZ) \weightof{z}$. Define also $\rhoX = \max\{1 + \epsForWeightXoverZ, \weightof{x}/\weightof{z}\}$ such that $(\rhoX - 2\epsForWeightXoverZ)\weightof{z} \le \weightof{x} \le \rhoX \weightof{z}$. Now, consider the line segment from $x$ to $z'$, the first point of intersection with $B$. Assuming that $x$ does not connect to $z$, the projection of this line segment along the weight dimension must cover distance at least $\dmin{}$ where $\dmin{}$ satisfies
    \begin{displaymath}
        (\Di{i} + \dmin{})^d \ge (\rhoX - 2\epsForWeightXoverZ) \weightof{z}^2 \implies \dmin{} \ge (\rhoX - 2\epsForWeightXoverZ)^{1/d}\weightof{z}^{2/d} - \Di{i}.
    \end{displaymath}
    This is because of the lower bound on $\weightof{x}$ and by triangle inequality. Now, on the weight dimension, again by the lower bound on $\weightof{x}$ and since the first point of contact with $B$ has an upper bound on its weight coordinate, we see that the segment moves upward in weight at most $\weightof{z}(2^{d/2} - \rhoX + 2\epsForWeightXoverZ)$. And so, the weight over $d$-dimensional distance slope of the line segment is at most 
    \begin{displaymath}
        \frac{\weightof{z}(2^{d/2} - \rhoX + 2\epsForWeightXoverZ)}{ (\rhoX - 2\epsForWeightXoverZ)^{1/d}\weightof{z}^{2/d} - \Di{i}} = \frac{\weightof{z}(2^{d/2} -\rhoX + 2\epsForWeightXoverZ)}{ ((\rhoX - 2\epsForWeightXoverZ)^{1/d} - \Dzero)\weightof{z}^{2/d}}.
    \end{displaymath}
    Notice that for $\Dzero$ going to $0$, this expression approaches $\weightof{z}^{1 - 2/d}(2^{d/2} - \rhoX) / (\rhoX)^{1/d}$. Therefore, let us define $\epsForSlope$ such that the slope is at most $(1 + \epsForSlope)\weightof{z}^{1 - 2/d}(2^{d/2} - \rhoX) / (\rhoX)^{1/d}$ and let $\slopeVal$ refer to this value. Now, this means that $\weightof{y} \le \weightof{x} + \slopeVal \geomdist{x}{y}$. Because $x$ and $y$ share an edge, we have (recalling that $\weightof{x} \le \rhoX \weightof{z}$)
    \begin{displaymath}
        \weightof{x}(\weightof{x} + \slopeVal \geomdist{x}{y}) \ge \volbetween{x}{y} \implies \rhoX \weightof{z} ( \rhoX \weightof{z}  + \slopeVal \geomdist{x}{y}) \ge \volbetween{x}{y}.
    \end{displaymath}
    Now, after substituting in $\slopeVal$ and writing $\geomdist{x}{y} = \CForXYDistOverWZsquared \weightof{z}^{2/d}$, the above can be rewritten as 
    \begin{displaymath}
        \weightof{z}^{2}(\CForXYDistOverWZsquared^d - (1 + \epsForSlope)\CForXYDistOverWZsquared \rhoX^{1-1/d}(2^{d/2} - \rhoX) - \rhoX^2) \le 0.
    \end{displaymath}
Notice for now that this imposes an upper bound on $\CForXYDistOverWZsquared$. We will work for the rest of the proof with this upper bound abstractly and only in the end show that it is tight enough for our purposes. For the following, keep $\weightof{x}$ as is but assume that we reduce $\weightof{y}$ to $\volbetween{x}{y}/\weightof{x}$. This modification preserves the edge from $x$ to $y$ and cannot introduce the edge from $y$ to $z$ if it was not there before, so we may carry it out wlog. Recall that we have assumed wlog that $\Di{i} \le \geomdist{x}{y} \Dzero/(1 - \Dzero)$. Observe that $\geomdist{x}{z} \le \geomdist{x}{z'} + \Di{i}$, $\geomdist{y}{z} \le \geomdist{y}{z'} + \Di{i}$ and that $\geomdist{x}{y} = \geomdist{x}{z'} + \geomdist{z'}{y}$. Because $\weightof{x} \le \rhoX\weightof{z}$, if we have 
\begin{displaymath}
\geomdist{y}{z'} \le (\rhoX^{-1/d} - \Dzero/(1 - \Dzero)) \geomdist{x}{y},   
\end{displaymath}
then $y$ would connect to $z$. So, we may assume that instead
\begin{displaymath}
    \volbetween{x}{z} \le (1 - \rhoX^{-1/d} + 2\Dzero/(1 - \Dzero))^d \volbetween{x}{y}.
\end{displaymath}
Now, since we know $\weightof{y} = \volbetween{x}{y}/\weightof{x}$, we conclude (using $\geomdist{x}{y} = \CForXYDistOverWZsquared{} \weightof{z}^{2/d} $) that $\weightof{y} \le \weightof{z} \CForXYDistOverWZsquared^d/(\rhoX - 2\epsForWeightXoverZ)$. Therefore, $x$ will connect to $z$, provided $\CForXYDistOverWZsquared$ is below and separated away from $\rhoX ^{2/d}/(\rhoX^{1/d} - 1)$.\footnote{Assuming $\Dzero$ sufficiently close to $0$. Indeed, let us pretend for a minute that $\Dzero$ is actually zero. Then, we would notice that $\weightof{x}\weightof{z} = \rhoX \weightof{z}^2$. Moreover, $\volbetween{x}{z} \le (1-\rhoX^{-1/d})^d \volbetween{x}{y} = (1-\rhoX^{-1/d})^d \CForXYDistOverWZsquared^d \weightof{z}^2$. If we plug in $\CForXYDistOverWZsquared = \rhoX ^{2/d}/(\rhoX^{1/d} - 1)$, we conclude that $\volbetween{x}{z} \le \rhoX \weightof{z}^2 \le \weightof{x} \weightof{z}$, and so $x$ connects to $z$.} We return to the inequality that provides this upper bound, which is essentially
\begin{displaymath}
    \CForXYDistOverWZsquared^d - \CForXYDistOverWZsquared \rhoX^{1-1/d}(2^{d/2} - \rhoX) - \rhoX^2 \le 0.
\end{displaymath}
The goal now will be to show that this inequality guarantees the desired bound on $\CForXYDistOverWZsquared{}$. We will do so by showing that i) $\CForXYDistOverWZsquared{} < 4$ and ii) $\rhoX ^{2/d}/(\rhoX^{1/d} - 1) > 4$. Let us start with the latter. Compare with the function $x^2/(x-1)$ on the interval $[1, \sqrt{2}]$, which is minimized at $x=\sqrt{2}$, attaining the value $2/(\sqrt{2}-1) > 4$. Now, for the first claim, assume for the sake of contradiction that $\CForXYDistOverWZsquared{} \ge 4$. We will show that the left hand side of the inequality in the most recent display would then be positive.

Since $\rhoX \in [1, 2^{d/2}]$, we may instead inspect the expression
\begin{equation}\label{eq:equation1}
      \CForXYDistOverWZsquared^d - \CForXYDistOverWZsquared \rhoX(2^{d/2} - \rhoX) - \rhoX^2.
\end{equation}
First, we will argue that viewed as a function of $\CForXYDistOverWZsquared{}$, the above expression is increasing for $\CForXYDistOverWZsquared{} \ge 4$. Indeed, the derivative is $d \CForXYDistOverWZsquared^{d-1} - \rhoX(2^{d/2} - \rhoX)$, which is minimized for $\rhoX = 2^{d/2}/2$. This last expression is increasing in $\CForXYDistOverWZsquared$, and setting it to $4$, we get $\frac{1}{4}(d 4^d - 2^d) > 0$. Now, it remains to show that~\eqref{eq:equation1} is positive for $\CForXYDistOverWZsquared = 4$. The resulting expression is a quadratic in $\rhoX$, namely $3\rhoX^2 - 2^{d/2 + 2}\rhoX + 4^d$, which is at least $4^d - (4/3)2^d > 0$.
% \begin{displaymath}
        
% \end{displaymath}

% Let us first deal with the case $d = 2$.\textcolor{red}{Annoyingly, we need to assume $d \ge 2$. So we should say that the problem is solved for $d = 1$ and so on.} The expression then becomes $\CForXYDistOverWZsquared^2 - \CForXYDistOverWZsquared 2(2 - 1) - 4$. This is a simple quadratic, whose largest root is $1 + \sqrt{5} \approx 3.23 < 3.41 \approx (1 - 2^{-1/2})^{-1}$. To deal with general $d$, one observes that for $\CForXYDistOverWZsquared \ge 1 + \sqrt{5}$ the difference $[\CForXYDistOverWZsquared^d - \CForXYDistOverWZsquared 2^{d/2}(2^{d/2} - 1) - 2^d] - [\CForXYDistOverWZsquared^2 - \CForXYDistOverWZsquared 2(2 - 1) - 4]$ is non-negative. To see why, consider the change of variable $y = 2^{d/2}$. Then, with $\logC = \log_{2}{\CForXYDistOverWZsquared{}}$, we study the function $g(y) = y^{2\logC} - (\CForXYDistOverWZsquared{ + 1})y^2 + \CForXYDistOverWZsquared{} y$. We then have $g''(2) = 2\logC (2\logC - 1) 2^{2\logC - 2} -2 (\CForXYDistOverWZsquared{} + 1) \ge \log_2{(1 + \sqrt{5})} \CForXYDistOverWZsquared^2/2 -2(\CForXYDistOverWZsquared{} + 1) \ge 0$ for $\CForXYDistOverWZsquared{} \ge 1 + \sqrt{5}$. Similarly, one can show that $g'(2) \ge 0$. Because $g''$ is increasing, we conclude that $g(y)$ is also increasing in the interval $[2, \infty)$, which shows the claim. Therefore, in all cases we have shown that one of the edges from $z$ to $x$ or $y$ must exist, and the proof is finished.

\end{proofEnd}

\begin{figure}[t]
\centering
\includegraphics[scale=0.85]{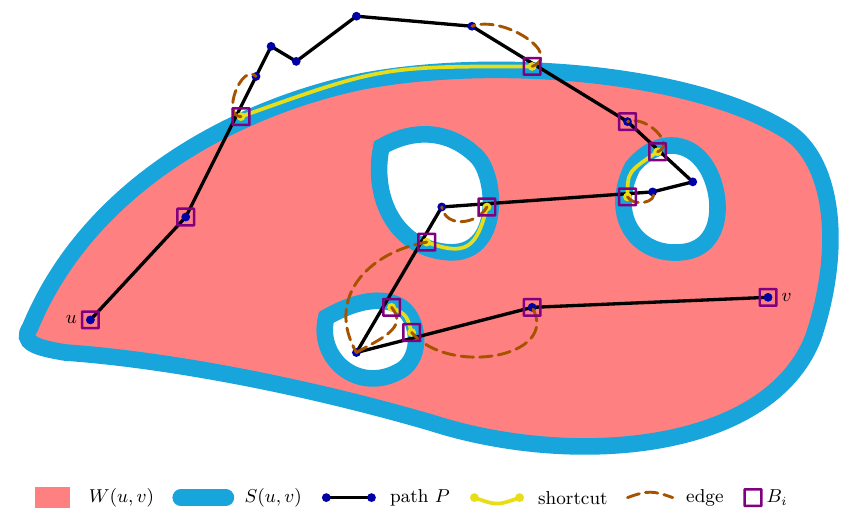}
\caption{Pictorial overview of the proof of~\Cref{lem:pathExistenceHighDensity}. The original path $P$ initially exits $\Wof{u}{v} \cup \Sof{u}{v}$ and reenters it without visiting another hole in the meantime. For this excursion,~\Cref{lem:boxcrossinglemma} provides two edges (dashed orange) we can use to ``anchor'' the path to the outer boundary. Afterwards, the path visits three consecutive holes without intersecting $\Wof{u}{v} \cup \Sof{u}{v}$. Bridges connecting the visible boundaries of these holes are found using~\Cref{claim:edgesBetweenHoles}. We finally identify a sequence of boxes $B_0, B_1, \ldots, B_k$ which are all in $\Wof{u}{v} \cup \Sof{u}{v}$ and contain a path from $u$ to $v$. In more detail, any two consecutive boxes contain vertices with graph distance at most $3$. For this, the worst case is in passing from the first hole to the second, where we start from the box on the first boundary, hop to one endpoint of the crossing edge, take the edge from one hole to the next and then finally hop from the other endpoint to the box on the second boundary. The boxes $B_0, B_1, \ldots, B_k$ are comprised of the violet ``checkpoint'' boxes and of the boxes one encounters when tracing the yellow shortcuts between them in $\Sof{u}{v}$.}
\label{fig:overall_proof}
\end{figure}

Let us now describe on a high level the rest of the proof. See also~\Cref{fig:overall_proof} for a visual aid to the following arguments. Consider the sequence of boxes induced by a shortest path between $u$ and $v$. That is, for each vertex in the path, write down the box in which it is contained. For later use, let us give a proper definition.
\begin{definition}[Box-footprint]
    Let $P = u_0, u_1, \ldots, u_k$ be a path in the GIRG graph from $u$ to $v$. We call the sequence $\boxof{u_0}, \boxof{u_1}, \ldots, \boxof{u_k}$ the \emph{box-footprint} of $P$. 
\end{definition}
This sequence need not be a walk in $\Gplus{}$. The crucial property is that a path from $u$ to $v$ is somehow contained in it. Our goal will be to transform this sequence in such a way that no boxes outside $\Wof{u}{v} \cup \Sof{u}{v}$ are contained in it, while preserving the path property. We have already discussed the general idea behind such modifications and depicted it in~\Cref{fig:2DBoundaries}. We will use~\Cref{lem:boxcrossinglemma} to establish ``anchors'' in the visible boundaries of $\Wof{u}{v}$ and use their connectedness to show the existence of paths between them. However, even the right panel of~\Cref{fig:2DBoundaries} still does not impart the full picture. There, we see two excursions being patched, one in the outside region and one inside the donut hole. But the path could also jump between holes before eventually reentering $\Wof{u}{v} \cup \Sof{u}{v}$ (as in~\Cref{fig:overall_proof}), in which case our argument temporarily breaks down. Fortunately, this then implies that some edge of the path must cross \emph{both} boundaries. Using~\Cref{lem:boxcrossinglemma} for the two boxes guaranteed to be crossed (each in one of the boundaries), we conclude that the two boundaries are actually still ``connected'' in the GIRG graph, in the sense that they contain vertices $u_1$ and  $u_2$ respectively which are separated by a constant number of edges in the GIRG graph. So, no matter how erratic an excursion might look, the set of visited ``holes'' is sufficiently well-knit, and therefore we can still find a sequence of boxes preserving the path. 
\jl{Hm, we don't need to necessarily change this before the deadline. But perhaps a clearer way to phrase it is: 1) Start with a u-v-path $P$. 2) Now consider the square $G^2$ of $G$, i.e., for every two vertices in distance $2$ we add an edge. Note that this cuts distances at most by a factor $2$. We will construct a short $u$-$v$ in $G^2$. 3) Consider any edge $e=\{x,y\}$ in $P$, where at least one endpoint $x$ is not in $W\cup S$. Then there is a vertex $z\in S$ that is adjacent to both $x$ and $y$ in $G^2$. We augment the path $P$ into a walk in $G^2$ by replacing the edge $e$ by the path $x - z - y$ of length $2$. 3) Now the walk in $G^2$ alternates between subpaths that are completely in $W\cup S$ and that are completely in the complement. The latter subpaths are preceded and followed by two vertices in $S$, and both those vertices are visible from the same box (we can take the box of any vertex on the subpath). 4) By 2.8 and 2.9, we may replace each subpath in the complement of $W\cup S$ by a path that stays within $S$. Hence, we have found a $u$-$v$-path in $G^2$ that stays within $W\cup S$. 6) We can shortcut the resulting path, since we never need more than two vertices in the same box. Assume there are three vertices $u_1,u_2,u_3$ in the same box on the path, then the edge $\{u_1,u_3\}$ exists, so we can shortcut the vertex $u_2$. Hence, there exists a $u$-$v$-path in $G^2$ of length at most $2|W\cup S|$. 7) Hence, there exsists a $u$-$v$-path in $G$ of length at most $4|W\cup S|$.}

Let us now initiate a formal treatment of the above paragraph. We start with a definition.

\begin{definition}[Hole]
    Consider two vertices $u, v$. A \emph{hole} is a connected component of $\G{} \setminus\Wof{u}{v}$. The \emph{visible boundary} of a hole $H$ is $\visibleboundary{x}{\Wof{u}{v}}$, for some $x$ in $H$. The precise choice of $x$ is irrelevant. Notice that all boxes in the visible boundary of $H$ are active by construction of $\Wof{u}{v}$.
\end{definition}

We now wish to establish the first ``anchor'' when the path from $u$ to $v$ initially exits $\Wof{u}{v} \cup \Sof{u}{v}$. To do so, we first need another definition and a claim that will be useful throughout the rest of the proof.

\begin{definition}[Box-shadow]
    Let $u, v$ be two vertices connected by an edge $e$. Let $S$ be the line segment in $\weightgeometry{}$ connecting the images of the vertices. We call the sequence of boxes intersected by $S$ (which is a walk in $\Gplus{}$) the \emph{box-shadow} of $\{u, v\}$ and denote it by $\boxshadow{e}$.
\end{definition}

\begin{theoremEnd}[normal, text link=]{claim}\label{claim:boundaryCrossings}
    Let $B_W$ be a box in $\Wof{u}{v} \cup \Sof{u}{v}$ and $B_H$ be a box in some hole $H$. Any walk in $\Gplus{}$ from $B_H$ to $B_W$ intersects the visible boundary of $H$.
\end{theoremEnd}
\begin{proofEnd}
    Consider the boxes $B_H = B_0, B_1, \ldots, B_k = B_W$ of such a walk. Because $B_W$ is in $\Wof{u}{v}$ or is a box that is $\Gplus{}$-adjacent to $\Wof{u}{v}$, there exists a minimum $i$ such that $\graphdist{\Gplus{}}{B_i}{\Wof{u}{v}} = 1$. We claim that $B_i$ is in the visible boundary of $H$. By construction, $B_i$ is $\Gplus$-adjacent to $\Wof{u}{v}$. It then suffices to show that $B_i$ is $\G$-connected to $B_H$, avoiding $\Wof{u}{v}$. Now, this follows because for every edge $\{B_j, B_{j + 1}\}$ of $\Gplus{}$ in the assumed walk with $j \le i - 1$, we know that $\graphdist{\Gplus{}}{B_j}{\Wof{u}{v}}$ is larger than 1. Therefore, $B_j$ and $B_{j + 1}$ are connected in $\G{} \setminus \Wof{u}{v}$. This follows from~\Cref{claim:shortcutGplus}. Indeed, the claim provides us with a path from $B_j$ to $B_{j+1}$ using only edges from $\G{}$ and where all the boxes satisfy have distance in $\Gplus{}$ at most $1$ from $B_{j}$. Therefore, they have distance at least $1$ from $\Wof{u}{v}$, i.e.\ they belong to the vertex set of $\G{} \setminus \Wof{u}{v}$.
\end{proofEnd}

Thus, to establish our first anchor, we identify the first box in the box-footprint that is not in $\Wof{u}{v} \cup \Sof{u}{v}$. It therefore belongs to some hole $H$. Now, the box just before is in $\Wof{u}{v} \cup \Sof{u}{v}$. Therefore, the box-shadow of the corresponding edge must intersect the visible boundary of $H$ by~\Cref{claim:boundaryCrossings}. The box found to intersect must be active, and therefore contains a vertex connected to some endpoint of said edge by~\Cref{lem:boxcrossinglemma}.

We have our first anchor, but we need to connect it all the way through to the next box of the box-footprint which will be in $\Wof{u}{v} \cup \Sof{u}{v}$ (such a box is guaranteed to exist because $v$ is in $\Wof{u}{v}$). To this end, we introduce the following claim.

\begin{theoremEnd}[restate, text link=]{claim}\label{claim:edgesBetweenHoles}
    Let $u_{H_1}, u_{H_2}$ be two vertices connected by an edge $e$ and suppose that the boxes containing them belong to different holes $H_1, H_2$. Then, there exists a vertex $v_1$ in some box $B_1$ in the visible boundary of $H_1$ and a vertex $v_2$ in some box $B_2$ in the visible boundary of $H_2$ with $\graphdist{GIRG}{v_1}{v_2} \le 3$. 
\end{theoremEnd}
\begin{proofEnd}
    It suffices to show that $\boxshadow{e} = \boxof{u_{H_1}}, \ldots, \boxof{u_{H_2}}$ intersects both visible boundaries of $H_1$ and $H_2$, since then these boxes where the intersections happen are active and thus they contain vertices which are connected through the edge $e$. We argue that there must exist some box $B_l$ in this box-shadow which has $\graphdist{\Gplus{}}{B_l}{\Wof{u}{v}} = 1$ and we choose the first such box. Then, this box is in the visible boundary of $H_1$, by the same argument as in~\Cref{claim:boundaryCrossings}. Similarly we can show intersection with the visible boundary of $H_2$.

    To see that indeed $B_l$ must exist, assume for the sake of contradiction that all boxes in $\boxshadow{e}$ have distance to $\Wof{u}{v}$ larger than 1. Then, by~\Cref{claim:shortcutGplus}, $\boxof{u_{H_1}}$ would be connected to $\boxof{u_{H_2}}$ in $\G{} \setminus \Wof{u}{v}$, which is a contradiction, as they belong to different holes.
\end{proofEnd}

We now have all the tools necessary to prove the main lemma of this subsection.

\begin{theoremEnd}[normal, text link=]{lemma}\label{lem:pathExistenceHighDensity}
    Suppose $u$ and $v$ are connected in the GIRG graph. There exists a sequence of boxes $B_0, B_1, \ldots, B_k$, entirely contained in $\Wof{u}{v} \cup \Sof{u}{v}$, with the following property. For every $i$, there is a vertex $u_i$ in $B_i$ with $u_0 = u$ and $u_k = v$ such that $\graphdist{GIRG}{u_j}{u_{j+1}} \le 3$ holds for all $j \in {0, \ldots k-1}$. In particular, $\graphdist{GIRG}{u}{v} \le \Oh*{|\Wof{u}{v}|}$. 
\end{theoremEnd}

\begin{proofEnd}
    Let $P = u'_0, u'_1, \ldots, u'_{k'}$ be the path assumed to exist between $u$ and $v$, and let $B'_0, B'_1, \ldots, B'_{k'}$ be its box-footprint. Let $i < j - 1$ be such that $B'_i$ and $B'_j$ are in $\Wof{u}{v} \cup \Sof{u}{v}$, but all $B'_l$ for $i< l < j$ are not. Let $H_{init}$ be the hole which contains $B'_{i + 1}$ and let $H_{fin}$ be the one which contains $B'_{j - 1}$ (with the possibility that $H_{init} = H_{fin}$). We will show three things that let us replace any such subsequence $B'_{i + 1}, \ldots, B'_{j - 1}$.
    \begin{enumerate}
        \item The visible boundary of $H_{init}$ contains a box $B_{init}$ with a vertex $u_{init}$ which satisfies $\graphdist{GIRG}{u_{init}}{u_i'} \le 3$, where $u_i'$ is the vertex which generated $B'_i$.
        \item A completely analogous statement holds for $H_{fin}$ and $B'_{j}$.
        \item There exists a sequence $B_{init}, \ldots, B_{fin}$ satisfying the properties listed in the statement of the lemma (treating $u_{init}$ and $u_{fin}$ as the start and end vertices).
    \end{enumerate}

    For the first statement, we use~\Cref{claim:boundaryCrossings}, providing the box-shadow of the edge $\{u'_i, u'_{i + 1}\}$ as the walk to be intersected. The bound on $\graphdist{GIRG}{u_{init}}{u_i'}$ follows because $u_{init}$ is the vertex in the active box found in the intersection and therefore it connects to either $u'_i$ or $u'_{i+1}$ by~\Cref{lem:boxcrossinglemma} (actually the bound here is 2 rather than 3). The second statement follows completely analogously.

    For the third statement, we employ~\Cref{claim:edgesBetweenHoles}. Indeed, in the subsequence $B'_{i + 1}, \ldots, B'_{j - 1}$, one may visit a sequence of holes $H_1, H_2, \ldots, H_h$. \Cref{claim:edgesBetweenHoles} then guarantees that for any two consecutive holes, their visible boundaries contain representative boxes that are (in 3 steps) connected through GIRG vertices. Therefore, a suitable sequence of boxes from $B_{init}$ to $B_{fin}$ exists, since the visible boundaries are themselves connected, comprised solely of active boxes and adjacent boxes have vertices that induce cliques in the GIRG graph.
\end{proofEnd}

\subsection*{Bounding $|\Wof{u}{v}|$.}
The final step in bounding the diameter is controlling the size of the inactive region itself. We will do this with a simple path-counting argument. For some $\Wof{u}{v}$ to grow large, there must occur large connected sets of generally inactive boxes. However, because the degree of the box-adjacency graph is bounded, there are exponentially many (in $k$) such possible sets of size $k$ and the probability of any fixed one being mostly inactive decays exponentially in $k$. Moreover, the strength of the decay depends on $\density$, whereas the scale of the combinatorial explosion does not. This Peierls-type idea is essentially the same as in~\cite[Lemma 13]{muller2019diameter}.

\begin{theoremEnd}[restate, text link=]{lemma}\label{lem:pathcountinglemma}
    Assume $\tau \le 3$ and that $\density$ is large enough. Then, with high probability we have  for some constant $C$ that $|\Wof{\boxof{1}}{\boxof{2}}| \le C\log{n}$ for all pairs $\boxof{1}, \boxof{2}$ of boxes.
\end{theoremEnd}

\begin{proofEnd}
    Fix a single pair $\boxof{1},\boxof{2}$. Suppose that for these boxes, $|\Wof{\boxof{1}}{\boxof{2}}| > C\log{n}$. For large enough $C$, we can conclude that at least half the boxes in $\Wof{u}{v}$ must be inactive (some are included by virtue of being part of $\Lof{u}{v}$ and these need not be inactive).

    In~\cite[pp.~129--130]{bollobas2006art}, it is shown that in a graph with $n$ vertices and maximum degree $\Delta$, there are at most $n(e(\Delta-1))^k$ connected subsets of size $k$. We have $\Delta = 3^{d}-1$. Then, there are (by a gross overestimation) at most $2^{k}$ ways to choose a subset that should be fully inactive. Each such subset (assumed to be at least half of the whole set) is fully inactive with probability at most $p_{in}^{k/2}$, where $p_{in}$ is an upper bound on the probability of a single box not containing a vertex. Since $\tau \le 3$, this probability is maximized at the lowest-level boxes. By increasing $\density$, we can assume $p_{in}$ is as low as desired. In particular, for some large enough $\density$, we have that the probability that there exists a half-inactive connected subset of $k$ boxes is at most $ne^{-k}$, which is $\oh{1}$ for large enough $C$.
\end{proofEnd}

\section{Upper bound: \texorpdfstring{$\tau < 3$}{τ < 3} and arbitrary \texorpdfstring{$\density$}{λ}.}\label{sec:any_density}
In this section we will forego the assumption that $\density$ is large enough and make do with $\tau < 3$.

Thus, we can no longer assume that each box is active with probability arbitrarily close to 1. We can now only do that for boxes above a certain level.\footnote{This is because the number of vertices in a box at level $i$ is Poisson distributed with mean $\Th*{2^{id(1 - (\tau - 1)/2)}}$, which rapidly increases with growing $i$.} We now have to somehow deal with the boxes of lower levels. To manage this, we group these boxes into what we will call ``towers''.\footnote{This idea is similar in spirit to the renormalization carried out in~\cite{muller2019diameter}, but requires a different approach. There, the renormalization is essentially reliant on the fact that the percolation threshold for 1D site percolation is $1$, which is no longer true in higher dimensions.} A tower simply consists of the union of a box and all boxes ``below'' it.
\begin{definition}[Towers]
    Given a level $i$, we define the towers of our tessellation as:
    \begin{displaymath}
        \towers{} \;=\; 
\left\{
    \left(\prod_{k=1}^d 
  \left[
    2^i(j_k - 1)\Dzero,\; 2^i j_k \Dzero
  \right)\right) \times \left[1, (2^{d/2})^{i + 1}\right)
  \;\Big|\;
  1 \le j_k \le 2^{\tessPower{0} - i}
\right\}.
    \end{displaymath}
\end{definition}

Notice that the sole difference between $\tessLevel{i}$ and $\towers{}$ is the lower end of the weight range.

We now define what it means for a tower to be active. 

\begin{definition}[Active tower]\label{def:activetower}
    Fix some $\eps > 0$. Let $\tower{0} \in \towers{}$ with $\twmax{}$ the upper end of the weight range of $\tower{0}$. Ensure that $\twmax{}^{\eps}$ is an integer power of $(\sqrt{2})^d$ (for technicalities below) by potentially decreasing $\eps$. Let $\tower{1}, \tower{2}, \dots, \tower{3^d - 1}$ be the towers neighboring $\tower{}$, $\fort{} = \bigcup_{i = 1}^{3^d-1}\tower{i}$ and let $G_{\fort{}}$ be the graph induced by vertices in $\fort{}$. We say that $\tower{0}$ is \emph{active}, if:
    \begin{enumerate}
        \item Any box (as in~\Cref{def:canonboxpath}) that is part of $\fort{}$ and whose lower weight range is at least $\twmax{}^{\eps}$ is active.\footnote{Note that this implies the next property.}
        \item In $G_{\fort{}}$, all vertices of weight at least $\twmax{}^{\eps}$ are in the same connected component $\comp$.
        \item In $G_{\fort{}}$, for any component $\comp' \neq \comp$, if $u, v \in \comp'$, then $\geomdist{u}{v} \le \twmax{}^{\Th*{\eps}}$.
    \end{enumerate}
\end{definition}

For the rest of this section, we redefine $\Lof{u}{v}, \Wof{u}{v}, \Sof{u}{v}$ in the following way. We simply replace the notion of boxes by elements which can be either towers or boxes of level above the tower cutoff. This still defines essentially the same graphs $\G{}, \Gplus{}$, up to a reduction in the number of levels. The definition of $\Lof{u}{v}$ is still based on the spanning tree of $\G{}$, but we consider as start/end vertices in said graph either the tower containing $u$/$v$ or the boxes which contain them if their weights are large enough. The definitions of $\Wof{u}{v}, \Sof{u}{v}$ simply inherit the effects of the change in $\Lof{u}{v}$. The goal will now be to re-establish the ``shortcuts'' through $\Sof{u}{v}$ and again show that $\Wof{u}{v}$ does not grow too large.

To that end, we first show that the higher the tower, the more likely it is to be active.

\begin{theoremEnd}[normal, text link=]{lemma}
Let $\tower{0} \in \towers{}$ be a tower of level $i$. Then, 
$$\Prob*{\text{$\tower{0}$ is active}} \ge 1 - \oh*{1},$$
where the $\oh{\cdot{}}$ is with respect to $i$.
\end{theoremEnd}

\begin{proofEnd}
It suffices to show that properties $1$ and $3$ fail to be satisfied each with $\oh{1}$ probability. For the first one, notice that the probability for any of the boxes of interest to be inactive is maximized for the lowest-level ones. Therefore, a fixed box is inactive with probability at most $\Exp*{-\twmax{}^{\Th*{\eps{}}}}$. It is not hard to see that there are polynomially many (in $\twmax{}$) boxes.\footnote{Because each box has $2^d$ ``children'' below it and we reach the lowest level boxes within $\Oh*{\log{\twmax{}}}$ iterations.} Thus, a union bound shows the claim.

Now, for property $3$, we will expose $G_{\fort{}}$ in two steps. First, reveal all vertices of weight at most $\twmax{}^{\eps}$ and the edges between. We then reveal the higher-weight vertices and all remaining edges. Up to a $\oh{1}$ failure probability (from property $1$), in the second step we only merge components from the first step to $\comp$. Therefore, it suffices to show that any component $\comp'$ from the first step violating property $2$ will merge with $\comp$ in the second step.

Indeed, let $\comp'$ be such a component and let $u, v$ be the vertices satisfying $\geomdist{u}{v} \ge \twmax{}^{C \eps}$ for a constant $C$ to be determined. Consider a shortest path between $u$ and $v$. This path must contain at least $\twmax{}^{C\eps - 2\eps}$ edges, since each edge can only cover a geometric distance at most $\twmax{}^{2\eps}$. Consider every other vertex (i.e.\ the first, the third, the fifth and so on) in this path and place a ball of radius $1/2$ around each one. If any pair of these balls would intersect, we could shortcut the path, which would contradict its minimum length. Therefore, we conclude that a subset of the vertices in the path cover a geometric area of volume $\Om*{\twmax{}^{C\eps - 2\eps}}$. Now, let $A$ be that area. If in the second step we reveal any vertex in $A$, the component $\comp'$ will merge with $\comp$ (as we will have an edge from some vertex of the path to one vertex in $\comp$). Now, the number of vertices of weight at least $\twmax{}^{\eps}$ in $A$ is Poisson distributed with expectation $\Om*{\twmax{}^{C\eps - 2\eps -(\tau - 1)\eps}}$. It follows that for large enough $C$ (in particular as long as $C > \tau + 1$), the probability that none are revealed is at most $\Exp*{-\twmax{}^{\Th*{\eps{}}}}$. This argument was for a singular $\comp'$. However, there can be at most polynomially many (in $\twmax{}$) candidates, so again a union bounds suffices. This bound on the number of components is because the total volume of $\fort{}$ is polynomial in $\twmax{}$, and each $\comp'$ would need to disjointly occupy at least some constant volume (much like in the shortest path argument). 
\end{proofEnd}

As a first observation towards re-establishing the $\Sof{u}{v}$ ``shortcuts'', let us remark that a version of~\Cref{lem:boxcrossinglemma} still holds when we replace some box levels by towers. In more detail, if it so happens that an edge from $x$ to $y$ intersects an active \emph{box} of not too low weight (i.e.\ just satisfying the \emph{weight} condition of property $1$ in~\Cref{def:activetower}), one of these vertices must connect to a vertex in the intersected box, since it is active. This follows verbatim by~\Cref{lem:boxcrossinglemma}.

\begin{definition}[High-weight box]
Fix a tower level $i$ and an $\eps > 0$ as in~\Cref{def:activetower}. Then, let $\twmax{}$ be as in~\Cref{def:activetower}. We say that a box of the tessellation is \emph{high-weight} if the lower limit of its weight range is at least $\twmax{}^{\eps}$.  
\end{definition}

% \begin{corollary}
%     Let $\Dzero$ in the definition of the tessellation be small enough. Let $x, y$ be two vertices connected by an edge intersecting $\Sof{u}{v}$ in some \emph{high-weight} box. Then, there exists a vertex $z$ in some \emph{high-weight} box of $\Sof{u}{v}$ which is adjacent to either $x$ or $y$.
% \end{corollary}
% \begin{proof}
%     Follows by the proof of~\Cref{lem:boxcrossinglemma}.
% \end{proof}

With those boxes out of the way, the only other option left for the $u-v$ path to exit $\Wof{u}{v} \cup \Sof{u}{v}$ is through the low-weight boxes of the towers. We will now show that this is also impossible without connecting to the largest component $\comp$ of some $G_{\fort{}}$. We first deal with the case where an edge possibly starts inside $\Wof{u}{v}$ and ends up somewhere outside of $\Wof{u}{v} \cup \Sof{u}{v}$. We show that this just cannot happen without intersecting a high-weight box of $\Sof{u}{v}$.

\begin{theoremEnd}[restate, text link=]{claim}\label{claim:highWeightBoxCrossing}
    There is no edge $e$ from a vertex $x$ to $y$ with $x$ in $\Wof{u}{v}$ and $y$ not in $\Wof{u}{v} \cup \Sof{u}{v}$ such that this edge does not intersect any high-weight box of $\Sof{u}{v}$. Moreover, the intersected box is guaranteed to be part of the visible boundary of the hole in which $\boxof{y}$ belongs.
\end{theoremEnd}

\begin{proofEnd}
    Assume without loss of generality that $\weightof{x} \le \weightof{y}$. Let $\twmax{}$ be as in~\Cref{def:activetower}. Consider the box-shadow $\boxshadow{e} = B_1, B_2, \ldots, B_k$. This is a $\Gplus{}$-walk from $\boxof{x}$ to $\boxof{y}$. Let $i < j - 1$ be a maximal pair such that $B_i \in \Wof{u}{v}$, $B_j \notin \Wof{u}{v} \cup \Sof{u}{v}$ and $B_l \in \Sof{u}{v}$ for all $i < l < j$. We must have $\graphdist{\Gplus{}}{B_i}{B_j} > 1$, and so any two vertices in these boxes must be separated in $d$-dimensions by at least a distance $\Om{\twmax{}^{2/d}}$. The part of the line segment travelling between two such vertices and through the boxes $B_l$ stays below weight $\twmax{}^{\eps}$ until it reaches $B_j$, otherwise it would intersect an active box. Therefore, this implies that the weight/distance slope of the edge's line segment is at most $\rho = C'(\twmax{}^{\eps} - \weightof{x})/\twmax{}^{2/d}$. Now, we may write $\weightof{y} = \weightof{x} + \rho \geomdist{x}{y}$. We know that $\volbetween{x}{y} \le \weightof{x} \weightof{y} = \weightof{x}(\weightof{x} + \rho \geomdist{x}{y})$. Since $\weightof{x} \rho < 1$ for small enough $\eps$ and/or large enough $\twmax{}$ and we also have $d \ge 1$, the validity of this inequality is monotone wrt $\geomdist{x}{y}$, in the sense that larger values for the latter can only falsify the inequality. We know that $\geomdist{x}{y}$ is $\Om*{\twmax{}^{2/d}}$, and so we must have for some constant $C > 0$ (depending on $\Dzero$ and importantly not on $\twmax{}$):
    \begin{displaymath}
        C^{2/d}\twmax{}^{2} \le \weightof{x}(\weightof{x} + C \twmax{}^{2/d} (\twmax{}^{\eps} - \weightof{x})/\twmax{}^{2/d}) \le \weightof{x}(\weightof{x}(1-C) + C\twmax{}^{\eps})
    \end{displaymath}
    Now, since $\weightof{x} \le \twmax{}^{\eps}$, we have a contradiction for $\eps < 1$ and  $\twmax{}$ large enough.
\end{proofEnd}

The final piece of the puzzle is to handle cases where $x$ is in $\Wof{u}{v}$ and $y$ is in a low-weight box of $\Sof{u}{v}$. We now need to switch viewpoint and instead of actually looking at an edge between $x$ and $y$ we identify $u_{in}$ as the last vertex to be in $\Wof{u}{v}$ and $u_{out}$ to be the first vertex after $u_{in}$ (on the path from $u$ to $v$) which is not in $\Sof{u}{v}$ (similar to the definition of $B_i$ and $B_j$ above). Between $u_{in}$ and $u_{out}$, all vertices are in $\Sof{u}{v}$. Such an arrangement must occur if the path from $u$ to $v$ is to ever exit $\Wof{u}{v} \cup \Sof{u}{v}$ (and in the case we are currently considering, at least one vertex between $u_{in}$ and $u_{out}$ must exist).

\begin{theoremEnd}[restate, text link=]{claim}
    It is impossible for a path to exist between a vertex $u_{in} \in \Wof{u}{v}$ to a vertex $u_{out} \notin \Wof{u}{v} \cup \Sof{u}{v}$ without some vertex in said path connecting (within its own tower) to a vertex in a high-weight box of $\Sof{u}{v}$. Moreover, the mentioned box (more accurately, the tower containing it) is guaranteed to be part of the visible boundary of the hole in which $\boxof{u_{out}}$ belongs.
\end{theoremEnd}

\begin{proofEnd}
    Let $u_1, u_2, \dots, u_k$ be the vertices between $u_{in}$ and $u_{out}$ in the path, all of which are wlog assumed to be in $\Sof{u}{v}$. We know that $\geomdist{u_{in}}{u_{out}} \ge C \twmax{}^{2/d}$, where $\twmax{}$ is as in~\Cref{def:activetower}, by similar arguments as in the proof of~\Cref{claim:highWeightBoxCrossing}. Suppose first that $\geomdist{u_{in}}{u_1} \ge \frac{1}{3} C \twmax{}^{2/d}$. In this case, notice that there must exist a vertex $z$ in a high-weight box within distance $C \twmax{}^{2/d}$ from $u_1$ with $\weightof{z} \ge 4^d \weightof{u_1}$, because $u_1$ is in a low-weight box of an active tower $\tower{a}$. By triangle inequality, $\volbetween{u_{in}}{z} \le 4^d \volbetween{u_{in}}{u_1}$. Therefore, $z$ must connect to $u_{in}$. So, for the rest of the proof we may assume $\geomdist{u_{in}}{u_1} \le \frac{1}{3} C \twmax{}^{2/d}$. 

    Let $G_{\fort{}}$ be as in~\Cref{def:activetower}, with $\tower{0}$ there replaced by $\tower{a}$. Assume wlog that $u_1$ is not in the same connected component as $z$ (the vertex in the top box of $\tower{a}$) in $G_{\fort{}}$. Moreover, let $v_{hop}$ be the first vertex among $u_{2}, \dots u_{k}, u_{out}$ which is not a vertex of $G_{\fort{}}$. One such vertex must exist, since otherwise we would conclude that $\geomdist{u_{in}}{u_{out}} \le \frac{1}{3} C \twmax{}^{2/d} + \twmax{}^{\Th*{\eps}} \le \frac{2}{3}C \twmax{}^{2/d}$. Let $v_{pre}$ be the vertex right before $v_{hop}$. Then, this means that either i) the \emph{geometric} position of $v_{hop}$ is outside of $\tower{a}$ and the neighboring towers (i.e.\ outside of $\fort{}$ as in~\Cref{def:activetower}) or ii) $\weightof{v_{hop}} \ge W$ which again implies i), if we (wlog) assume that $z$ does not connect to $v_{hop}$. In any case, we conclude that $\geomdist{v_{pre}}{v_{hop}} \ge \frac{1}{3} C \twmax{}^{2/d}$. By a similar argument as in the first paragraph, this implies that $z$ connects to $v_{hop}$, finishing the proof.
\end{proofEnd}

The above claims collectively prove the following, analogous to~\Cref{lem:pathExistenceHighDensity}.

\begin{corollary}\label{cor:UB1}
    Suppose $u$ and $v$ are connected in the GIRG graph. There exists a sequence of boxes/towers $X_0, X_1, \ldots, X_k$, entirely contained in $\Wof{u}{v} \cup \Sof{u}{v}$, with the following property. There exists a constant $C$ such that for every $i$, there is a vertex $u_i$ in $X_i$ with $u_0 = u$ and $u_k = v$ such that $\graphdist{GIRG}{u_j}{u_{j+1}} \le C$ holds for all $j \in {0, \ldots k-1}$. In particular, $\graphdist{GIRG}{u}{v} \le \Oh*{|\Wof{u}{v}|}$. 
\end{corollary}

\subsection*{Final steps in bounding the diameter.} Just like we did for the case $\tau \le 3$ with high enough $\density$, one can show that the inactive regions, where we now fix a large enough $i$ and a small enough $\eps$ and consider towers instead of boxes for the low levels, do not grow too large. A crucial difference is that the events $E_1$ and $E_2$ meaning that neighboring towers $\tower{1}$ and $\tower{2}$ are active are \emph{not independent}. However, in any set of $k$ towers/boxes, we can always identify a constant fraction of them that \emph{are} completely independent, and that fraction is roughly $1/(3^{d} - 1)$. This only forces us to drive the probability $p_{in}$ in~\Cref{lem:pathcountinglemma} even closer to $0$. This argument proves the following, analogous to~\Cref{lem:pathcountinglemma}.

\begin{corollary}\label{cor:UB2}
    Assume $\tau < 3$. Then, with high probability we have for some constant $C$ that $|\Wof{X_1}{X_2}| \le C\log{n}$ for all pairs $X_1, X_2$ of boxes/towers.
\end{corollary}

\section{Lower bound.}\label{sec:LB}
In this section we prove a matching logarithmic lower bound for the diameter when $\tau > 2$. The idea will be to identify $n^{\Om{1}}$ regions in the ground space and argue that each of them has an independent chance of producing a path component of logarithmic size. The success probability per trial will be $n^{-C'}$ with a constant $C'$ of our choice, and we show that at least one will succeed whp.

In more detail but still at an intuitive level, a central point is that the maximum weight found in the graph will whp be at most roughly $n^{1/(\tau - 1)}$, which is $o(n)$ for $\tau > 2$. The paths we will try to construct will consist of vertices with constant weight. Therefore, a single trial requires only about $n^{1/(\tau-1)}$ volume of space. Inside this region, we can locally decide whether there exists any unwanted vertex connecting to our under-construction path. Outside the region, no vertex can connect to the path under the aforementioned maximum weight constraint. 

Let us also provide some more details about how to decide existence and local disconnectivity of the path to the rest of the region. First of all, we will look for the path at the middle of the region. Geometrically, it will be a ``snaky'' path confined in a ball of radius $\Theta(\log^{1/d}{n})$ but still occupying a constant fraction of its volume (see~\Cref{fig:hilbert}). Assuming such a path exists, we now try to ensure that no other vertices connect to it. Fix an $r$ and consider the geometric topos with minimum distance to any of the vertices of the path in the interval $[r, 2r]$. Since the maximum distance between any two vertices in the path is $C \log^{1/d}{n}$, it has volume at most $(r + C \log^{1/d}{n})^d$. Now, for a vertex in this region to connect to any vertex in the path, its weight must be at least $\Om{r^d}$. Going over $r$ in powers of two, it can be seen that the expected number of such vertices is at most $C' \log{n}$, for a constant $C'$ that can be controlled by reducing $C$. Since this number is Poisson-distributed, the probability that it is equal to $0$ is at least $\Exp{-C' \log{n}} = n^{-C'}$. Assuming we can attempt such a trial $n^{C''}$ times, we simply need to ensure $C' < C''$. For context, $C''$ will roughly be $1 - 1/(\tau - 1)$, i.e.\ the number of disjoint volume $n^{1/(\tau - 1)}$ cubes we can find.

Now, let us describe a curve through the cubes of a uniform tessellation of a $d$-dimensional cube, which will serve as a ``skeleton'' for our path.

\begin{definition}\label{def:grayCurve}
    Let $R$ be a cube-shaped region in $d$ dimensions and tessellate it into $M^d$ identical cubes. Each such cube can be represented as a vector with $d$ entries from $\{0, 1, \ldots, M - 1\}$, encoding the geometric position of its center. Order these vectors according to the M-ary Gray code, i.e.\ every two consecutive vectors differ only in a single entry and by exactly $1$. We define the \emph{Gray curve} of level $M$ of $R$ as the piecewise linear curve going through the centers of the cubes in the order given by the M-ary Gray code.
\end{definition}

\begin{theoremEnd}[restate, text link=]{observation}\label{obs:awaycubes}
    For any two non-consecutive cubes visited by the Gray curve, the following holds. If $x_1, x_2$ are points on the curve in these cubes respectively, then $\geomdist{x_1}{x_2} \ge c_d s$, where $s$ is the side length of the cubes and $c_d$ is a constant depending on the dimension $d$. 
\end{theoremEnd}

\begin{proofEnd}
    Notice that each cube contains inside it one or two parts of the curve, both of which are line segments from the center of the cube to the center of some face of the cube, with length $s/2$. It suffices to show that such pairs of line segments do not come within distance smaller than $c_d s$. 

    If the centers of the cubes are separated by distance larger than $s$, the claim follows by triangle inequality (the set of possible distances is quantized). Therefore, for the rest of the proof we assume that the centers are exactly $s$ units away from each other, i.e.\ the cubes share a face. 

    Notice that there is a limited set of directions a line segment can have (of cardinality $2d$), seen as oriented away from the center of the cube. If the segments are both oriented straight towards the other cube's center, the cubes would be consecutive. Thus, they are not and the claim follows by the law of cosines.
\end{proofEnd}

We proceed with a claim that gives a lower bound on the probability of a cube of volume $n^{1/(\tau - 1) + \eps}$ to be locally good, using the object of~\Cref{def:grayCurve}.

\begin{figure}[ht]
  \centering
  \includegraphics[width=0.4\linewidth]{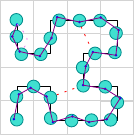} % adjust width as needed
  \caption{Logarithmic diameter lower bound construction for $d=2$. Inside a uniform tessellation of a box $R'$ of logarithmic volume, we consider the Gray curve (\Cref{def:grayCurve}) going through the centers of the boxes. Tracing along the curve, we place balls of radius 1 where we look for vertices of restricted weights. By definition, vertices in boxes not consecutive in the Gray curve cannot connect (depicted in dashed red). Other pairs of nearby balls may introduce small shortcuts (dashed pink).}\label{fig:hilbert}
\end{figure}

\newcommand{\centerDist}{\ensuremath{s_{cube}}}
\begin{theoremEnd}[restate, text link=]{claim}\label{claim:snakePathEvent}
    Let $C_1, \eps > 0$. Consider a cube-shaped region $R$ of volume $n^{1/(\tau - 1) + \eps}$ in the geometric space of the GIRG graph $G$ (with $V(R)$ the set of vertices in $R$) and let $E$ be the event that $G[V(R)]$ contains a component of diameter $C_1 \log{n}$ comprised solely of vertices of weight at most $3^d$. Then $\Prob{E} \ge \Exp{-C_2 \log{n}}$, for a $C_2$ that depends on $C_1$.
\end{theoremEnd}
\begin{proofEnd}
    Wlog translate $R$ such that its center is the origin. Let $R'$ be a smaller similar region of volume $C_1' \log{n}$ (with $C_1'$ sufficiently large compared to $C_1$), also centered at the origin. Tessellate $R'$ into identical cubes in the finest way (i.e.\ largest $M$ as in~\Cref{def:grayCurve}) possible such that the following holds.
    For any two non-consecutive cubes visited by the curve and any two corresponding points on the curve (within the mentioned cubes respectively), said points have distance at least $12$ from each other. Notice that by~\Cref{obs:awaycubes}, $M$ is at least $\Om{\log^{1/d}{n}}$, i.e.\ the cubes shrink down to constant volume.
    
    Now, consider balls of volume 1 (i.e.\ also having radius 1) placed on the Gray curve whenever the distance to the previous ball's center becomes $2$ (we start by placing one ball at the start of the curve). Note that each cube is assigned at least one ball, by~\Cref{obs:awaycubes} (or more simply by demanding that the side length of the tessellation is at least $4$). With at least the required probability, each such ball contains \emph{exactly} one vertex with weight in $[2^d, 3^d]$. Now, notice that the vertices of consecutive balls connect, since their distance is at most $4$ and the product of their weights is at least $4^d$. Moreover, there is no edge between vertices in balls inside non-consecutive cubes (in the Gray curve ordering), since the distance between them would be at least $12 - 1 - 1 = 10$, while the product of their weights is at most $9^d$. These arguments imply the existence of a \emph{shortest} path $P$ with length at least the number of cubes of the tessellation, which is in turn in the order of the volume of $R'$, i.e.\ logarithmic.

    It now remains to show that with good enough probability, this path is not shortcut when exposing the rest of the graph inside $R$. To this end, fix some $r = 2^l$ with $l$ a non-negative integer. Consider the geometric topos of points with minimum distance to a vertex of $P$ in the interval $[r, 2r]$. Since all vertices of $P$ are in $R'$, the volume of this topos is $\Oh{\max\{r^d, \log{n}\}}$. Now, to connect to a vertex of $P$, a vertex in said topos must have weight $\Om{r^d}$. The number of vertices in this topos with weight at least this lower bound is Poisson distributed with expectation $\Oh{ r^{d(1- \tau)}(\max\{r^d, \log{n}\})}$. Summing over all $l$, we observe that the overall expected number of vertices connecting to $P$ is $\Oh{\log{n}}$ and is also Poisson-distributed. Thus, the probability that it is $0$ is at least as desired, as the constant in Landau notation depends on $C_1$.
\end{proofEnd}

With the above claim at hand, we may now proceed to prove this section's theorem.

\begin{theoremEnd}[normal, text link=]{theorem}\label{theorem:LB}
    Assume $\tau > 2$. Then, whp the diameter of GIRG is $\Om{\log{n}}$.
\end{theoremEnd}
\begin{proofEnd}
    Fix a small enough $\eps > 0$ and observe that whp the maximum weight of the graph will be at most $n^{1/(\tau - 1) + \eps}$. Identify $\Th*{n^{1 - 1/(\tau - 1) - 2\eps}}$ disjoint cubes of volume $n^{1/(\tau - 1) + 2\eps}$.\footnote{The factor $2$ is introduced to avoid dealing with some constants in arguing that no connections exist from $R'$ to outside of $R$ from~\Cref{claim:snakePathEvent}.} Now, if the event of~\Cref{claim:snakePathEvent} succeeds for any one of these cubes, we have a witness to the claimed lower bound on the diameter, since the discovered component has large diameter within its own cube and does not connect to any vertex outside the cube, by virtue of the upper bound we have on the maximum weight and a lower bound on the distance between a vertex in said component to any vertex outside the cube (since the component is localized in $R'$ from the proof of~\Cref{claim:snakePathEvent}). By setting $C_2$ in~\Cref{claim:snakePathEvent} small enough compared to $1 - 1/(\tau - 1) - 2\eps > 0$, we conclude that the number of successful cubes is binomially distributed with expectation $n^{\Om{1}}$. Therefore, whp we have not only one but polynomially many components of logarithmic diameter in the graph.
\end{proofEnd}

% \begin{remark}
%     The above proof can be modified to also show that the giant component has logarithmic diameter. The only thing we need to change is to look for paths in~\Cref{claim:snakePathEvent} that instead of not connecting to any other vertex in $R$ connect to exactly one vertex of polylogarithmic weight (with e.g.\ at least the first half of the vertices of the path \emph{not} connecting to it, to maintain its length). Since all vertices above a certain polylogarithmic weight are part of the giant whp~\cite[Theorem 5.9]{bringmann2018averagedistancegeneralclass}, this would show the claim.
% \end{remark}

\section{Conclusion}\label{sec:Conclusion}
We have shown that the diameter of T-GIRG is $\Theta(\log n)$. A natural next question is whether the same holds for GIRG with positive temperature. The additional edges in this case are also called~\emph{weak ties}, and they have fundamental impact on many processes on GIRG, see for example~\cite{komjathy2023four,kaufmann2024rumour}. Our proof method breaks down in that case, mainly since Lemma~\ref{lem:boxcrossinglemma} is no longer true: a weak edge may cross an active box although both endpoints are too far away to reliably connect to any vertices in the box. Nevertheless, we conjecture that the diameter is still $\Theta(\log n)$ in this case, and a proof of this conjecture would be highly interesting.

\bibliographystyle{plainurl}% the mandatory bibstyle
\bibliography{bibliography}

\newpage

\appendix
\section{Upper bound: non-deterministic edges.}\label{sec:puzzle}
So far, we have assumed that an edge \emph{always} exists between vertices $x$ and $y$, provided $\weightof{x}\weightof{y} \ge \volbetween{x}{y}$. However, a very natural modification to the model is to only include such edges with some probability $\connprob < 1$. With the current definition of active boxes/towers, variants of~\Cref{lem:boxcrossinglemma} break down, since the required edge between $z$ (the vertex inside the intersected box as in~\Cref{fig:intersected_box}) and one of $x$ or $y$ may fail to exist.

The solution is to redefine what it means for a box to be active. To keep things simple, let us assume for the exposition of this generalization attempt that $\density$ is large enough, so we will not need to deal with towers. However, we should note that it is possible to do so at a cost of some additional technical overhead.

Since we want the desired edges to be present in the graph, we can (modulo some technical details) simply redefine an active box as one which satisfies the following. For \emph{every edge} $e = \{x, y\}$ intersecting the box, there exists a vertex $z_e$ in the box which satisfies $\min\{\graphdist{GIRG}{z_e}{x}, \graphdist{GIRG}{z_e}{y}\} \le C$ for some constant $C$. To maintain good connectivity within the box, we also require that the box itself induces a connected graph of constant diameter. The latter property is local to the box, but the former can in principle depend on the entirety of $\weightgeometry{}$ (due to the possible existence of high weight vertices far from the box). This renders futile an attempt to argue as in~\Cref{lem:pathcountinglemma} with the goal of bounding $\Wof{u}{v}$. 

However, independence is only used in~\Cref{lem:pathcountinglemma} as a means to conclude that $|\Wof{u}{v}|$ is at most logarithmic. We now present a different approach to proving the latter, which does not rely on independence, at least not in the same way. Keep in mind that we are simply reducing the task of proving the required size bound to proving a particular bound on the probability of an event that should, at least intuitively, be extremely unlikely. It still remains open to prove this bound.

Let us now formally redefine what it means for a box to be active in this new context, which is slightly different from the intuitive explanation above.

\begin{definition}[Active box]\label{def:activeBoxV2}
    Let $B$ be a box of the tessellation of $\weightgeometry{}$. Given the full realization of the GIRG graph $G_{GIRG}$, we say that $B$ is \emph{active}, if:
    \begin{enumerate}
        \item $B$ contains at least half of the expected number of vertices. 
        \item $G_{GIRG}[V_{local}]$ is connected and has diameter at most $C'$, where $V_{local}$ is the set of vertices within boxes at $\Gplus{}$-distance at most $C$ from $B$.
        \item For any edge from $x$ to $y$ in $G_{GIRG}$ that intersects $B$ with both $x$ and $y$ not in $V_{local}$, there exists some vertex $z$ in $B$ with $\min\{\graphdist{GIRG}{z}{x}, \graphdist{GIRG}{z}{y}\} \le C$.
    \end{enumerate}
\end{definition}
Notice that this definition suffices to recover the statement $\graphdist{GIRG}{u}{v} \le \Oh{|\Wof{u}{v}|}$. Also, note that the first two items of the definition induce only local dependencies between boxes. We therefore now examine the third item. Recall~\Cref{lem:boxcrossinglemma} and let $S(B)$ refer to the set of vertices which should connect to the arbitrary vertex $z$ assumed to exist in $B$. That is, for each edge $\{x, y\}$ that intersects the box at a point $z'$, one can identify from the proof a $x' \in \{x, y\}$ which satisfies $(\geomdist{x'}{z'} + \Di{i})^d \le \weightof{x'} \weightof{z}$, i.e.\ $x'$ has a chance to connect to $z$, regardless of where $z$ lies in $B$. We collect all such $x'$ in $S(B)$.

Now, let us partition $S(B)$ into $S_{high}(B)$ and $S_{low}(B)$ where an $x'$ is in $S_{high}(B)$ if and only if $\weightof{x'} \ge 2^d \weightof{z}$. We notice that any $x' \in S_{low}(B)$ satisfies $\graphdist{\Gplus{}}{B}{\boxof{x'}} \le C$ for large enough $C$.\footnote{First, notice that $\weightof{x'} \weightof{z} \ge \Di{i}^d$, meaning $\weightof{x'} \ge \Om{\weightof{z}}$. Therefore, the level of $\boxof{x'}$ is at most a constant below that of $B$ (and at most 2 above). By assumption we also know $\volbetween{x'}{z} \le \weightof{x'} \weightof{z} \le \Oh{\Di{i}^d}$. There are only a constant number of boxes $x'$ could belong to which satisfy both these constraints.} Therefore, such $x'$ are ``taken care of'' locally. The following claim provides a powerful handle on elements of $S_{high}(B)$.

\begin{claim}
    Let $x_1, x_2$ be elements of $S_{high}(B)$. Then, $\weightof{x_1} \weightof{x_2} \ge \volbetween{x_1}{x_2}$.
\end{claim}
\begin{proof}
    Notice that $\weightof{x_1}\weightof{x_2} \ge 2^d \weightof{x_1}\weightof{z}$. Assume wlog that $\geomdist{x_1}{z} \ge \geomdist{x_2}{z}$, which implies $\volbetween{x_1}{x_2} \le 2^d \volbetween{x_1}{z} \le 2^d \weightof{x_1}\weightof{z} \le \weightof{x_1} \weightof{x_2}$.
\end{proof}

The above claim shows that while faraway vertices that ``demand'' to be connected for the box to be active may exist, they must have a chance to connect in between themselves, allowing for new ways a constant-length path to the box might occur. This motivates the definition of the following ``puzzle''.

\begin{puzzle}
    Let $G$ be a graph, whose vertex set is comprised of disjoint sets $L$ and $R$. Each vertex of $L$ corresponds to a box $B$. Each vertex of $R$ corresponds to a vertex in some $S_{high}(B)$. An edge between $B \in L$ and $x \in R$ exists if $x \in S_{high}(B)$ (called a \emph{cross-edge}). There are no edges between vertices in $L$. An edge between two vertices in $R$ (called an \emph{$R$-edge}) exists if and only if their neighborhoods in $L$ overlap. 

    Keep each $R$-edge with probability $\connprob$ and each cross-edge with probability $p'$. The value of $p'$ is the probability that a vertex $x \in S_{high}(B)$ connects to \emph{any} vertex in $B$. Let $G'$ be the resulting graph. We define an event $E$ as follows:
    \begin{displaymath}
        E = \{\text{There exists some $B \in L$ which satisfies $\graphdist{G'}{B}{x} \le C$ for all $x \in S_{high}(B)$}\}.
    \end{displaymath}

    The goal of the puzzle is to show that $\Prob*{\neg E} \le \Exp{-f(p')|L|}$, for some growing $f(p')$ whose limit is infinity when $p'$ goes to $1$.
\end{puzzle}

Let us explain how this would show a logarithmic upper bound on $|\Wof{u}{v}|$. Recall that in the proof of~\Cref{lem:pathcountinglemma} we show that for $|\Wof{u}{v}|$ to grow large, there must exist a subset of it that is at least half the total size and also fully inactive. With good enough probability, at least a constant fraction of this set satisfies at least the first two items in~\Cref{def:activeBoxV2}. We therefore need to show that with good enough probability, at least one of these boxes will also satisfy the third item. This is precisely where the above puzzle comes into play. If one manages to prove this bound, then we just need to set $p'$ large enough to ``kill'' the combinatorial explosion of the different choices for $\Wof{u}{v}$ and its half-full inactive set, which grows as $\Exp{\Oh{|L|}}$. To drive $p'$ larger, we simply need to assume $\density$ is large enough so that each box has a sufficient number of vertices by the first item of~\Cref{def:activeBoxV2}.
\section{Deferred Proofs}
\printProofs{}

\end{document}